\documentclass[11pt]{amsart}
\usepackage{amsmath,amssymb,latexsym,soul,cite,amsthm,color,enumitem,graphicx,tikz, mathtools,microtype}
\usepackage[colorlinks=true,urlcolor=blue,citecolor=red,linkcolor=blue,linktocpage,pdfpagelabels,bookmarksnumbered,bookmarksopen]{hyperref}
\definecolor{cadmiumgreen}{rgb}{0.0, 0.42, 0.24}
\usepackage[english]{babel}
\usepackage{backref}
\usepackage[left=2.7cm,right=2.7cm,top=2.35cm,bottom=2.35cm]{geometry}

\numberwithin{equation}{section}

\newtheorem{theorem}{Theorem}[section]
\theoremstyle{plain}
\newtheorem{lemma}[theorem]{Lemma}
\theoremstyle{plain}
\newtheorem{proposition}[theorem]{Proposition}
\theoremstyle{plain}

\theoremstyle{definition}
\newtheorem{remark}[theorem]{Remark}

\newcommand{\N}{{\mathbb N}}

\newcommand{\R}{{\mathbb R}}
\newcommand{\Ex}{{{\rm Ex}}}
\newcommand{\eps}{\varepsilon}
\newcommand{\beq}{\begin{equation}}
\newcommand{\eeq}{\end{equation}}
\renewcommand{\le}{\leqslant}
\renewcommand{\ge}{\geqslant}

\newcommand{\w}{W^{s,p}_0(\Omega)}
\newcommand{\fpl}{(-\Delta)_p^s\,}
\newcommand{\ds}{{{\rm d}}_\Omega^s}
\def\Xint#1{\mathchoice
{\XXint\displaystyle\textstyle{#1}}%
{\XXint\textstyle\scriptstyle{#1}}%
{\XXint\scriptstyle\scriptscriptstyle{#1}}%
{\XXint\scriptscriptstyle\scriptscriptstyle{#1}}%
\!\int}
\def\XXint#1#2#3{{\setbox0=\hbox{$#1{#2#3}{\int}$ }
\vcenter{\hbox{$#2#3$ }}\kern-.6\wd0}}

\def\dashint{\Xint-}

\newenvironment{enumroman}{\begin{enumerate}

}{\end{enumerate}}

\title[Fine boundary regularity for the fractional $p$-Laplacian]{Fine boundary regularity for the degenerate\\ fractional ${\bf p}$\,-\,Laplacian}

\author[A.\ Iannizzotto, S.\ Mosconi, M.\ Squassina]{Antonio Iannizzotto, Sunra Mosconi, and Marco Squassina}

\address[A.\ Iannizzotto]{Department of Mathematics and Computer Science
\newline\indent
Universit\`a degli Studi di Cagliari, 
Viale L.\ Merello 92, 09123 Cagliari, Italy}
\email{antonio.iannizzotto@unica.it}

\address[S.\ Mosconi]{Dipartimento di Matematica e Informatica
\newline\indent
Universit\`a degli Studi di Catania, 
Viale A.\ Doria 6, 95125 Catania, Italy}
\email{mosconi@dmi.unict.it}

\address[M.\ Squassina]{Dipartimento di Matematica e Fisica
\newline\indent
Universit\`a Cattolica del Sacro Cuore,
Via dei Musei 41, 25121 Brescia, Italy}
\email{marco.squassina@unicatt.it}


\begin{document}

\begin{abstract}
We consider a pseudo-differential equation driven by the fractional $p$-Laplacian $\fpl$ with $s\in(0,1)$ and $p\ge 2$ (degenerate case), with a bounded reaction $f$ and Dirichlet type conditions in a smooth domain $\Omega$. By means of barriers, a nonlocal superposition principle, and the comparison principle, we prove that any weak solution $u$ of such equation exhibits a weighted H\"older regularity up to the boundary, that is, $u/\ds\in C^\alpha(\overline\Omega)$ for some $\alpha\in(0,1)$, ${\rm d}_\Omega$ being the distance from the boundary.
\end{abstract}

\maketitle

%

\section{Introduction and main result}\label{sec1}

\noindent
This paper is devoted to the study of some fine boundary regularity properties of the weak solution to the following problem:
\beq\label{dp}
\begin{cases}
\fpl u=f & \text{in $\Omega$} \\
u=0 & \text{in $\Omega^c$.}
\end{cases}
\eeq
Here, and throughout the paper, $\Omega\subset\R^N$ ($N>1$) is a bounded domain with a $C^{1,1}$ boundary $\partial\Omega$, $\Omega^c=\R^N\setminus\Omega$, $s\in\ (0,1)$, $p\in\ (1,\infty)$ are real numbers, and $f\in L^\infty(\Omega)$. The leading operator is the $s$-fractional $p$-Laplacian, defined as the gradient of the energy
\[J(u)=\frac{1}{p}\iint_{\R^N\times\R^N}\frac{|u(x)-u(y)|^p}{|x-y|^{N+ps}}\,dx\,dy\]
in the space
\[\w=\big\{u\in L^p(\R^N):\,J(u)<\infty,\,u=0 \quad \text{in $\Omega^c$}\big\}.\]
When restricted to conveniently smooth $u$'s, such operator can be rephrased pointwisely as
\[\fpl u(x)=2\lim_{\eps\to 0^+}\int_{B^c_\eps(x)}\frac{|u(x)-u(y)|^{p-2}(u(x)-u(y))}{|x-y|^{N+ps}}\,dy,\]
i.e., as a pseudo-differential operator of fractional order $s$ and summability power $p$, which for $p=2$ reduces to the Dirichlet fractional Laplacian $(-\Delta)^s$ (up to a multiplicative constant). For a deep discussion on various notions (weak, viscous and strong)  of solutions to \eqref{dp}, see \cite{KKL}. A useful comparison principle for $\fpl$ has been proved in \cite{LL}, a Hopf's lemma in \cite{DQ} and some strong comparison principles in \cite{J}, while its spectral properties are studied in \cite{FP,IS,LL}. 
\vskip2pt
\noindent
The interior regularity theory for problem \eqref{dp} is quite well developed. The linear case $p=2$ is quite classical and Schauder estimates are available in the form $f\in C^{\alpha}\Rightarrow u\in C^{2s+\alpha}$ whenever $2s+\alpha$ is not an integer (see \cite{RS}). In the general case $p\neq 2$ the situation is more involved. The first results are \cite{DKP, DKP1}, dealing with local regularity and Harnack inequalities when $f\equiv 0$ in \eqref{dp}. In the inhomogeneous case \cite{L1, KMS, BP, IMS1, IMS2} contain local H\"older regularity estimates under various integrability assumptions on $f$, however the dependance of the H\"older exponent is not specified and not optimal. The papers \cite{KMS2, S, BL} deal with the degenerate case $p\ge 2$ and show higher fractional differentiability of $u$ when fractional differentiability of the forcing term is assumed. In \cite{MM}  higher fractional differentiability is obtained for any $p>1$ under summability assumptions on $f$. Finally, still in the case $p\ge 2$, the optimal H\"older exponent for the solution of \eqref{dp} is obtained in \cite{BLS}, giving e.g.\ $u\in C^{p's}_{\rm loc}(\Omega)$ when $f\in L^{\infty}(\Omega)$ and $p's<1$.
\vskip2pt 
\noindent
The boundary regularity for problem \eqref{dp} is more delicate. As a comparison, consider its classical counterpart
\beq
\label{dpex}
\begin{cases}
-\Delta_p u=f &\text{in $\Omega$}\\
u\equiv 0&\text{on $\partial\Omega$},
\end{cases}
\eeq
(formally obtained by letting $s\to 1^-$ in \eqref{dp}). It is well known that, for example, $u\in C^{1,\alpha}_{\rm loc}(\Omega)$ whenever $f$ is bounded, and nothing more can be expected, regardless of the smoothness of $f$. This regularity can easily be extended up to the boundary, as follows. One straightens the boundary near $x_{0}\in \partial\Omega$ and consider the odd reflection of the resulting $u$: as it turns out, it solves a similar equation in a larger domain containing $x_{0}$ in its interior, therefore satisfying the previous local regularity estimates. The odd reflection trick then shows that in general the interior and boundary regularity for \eqref{dpex} coincide. Boundary regularity for a wider class of nonlinear {\em local} operators is proved in \cite{L}.
\vskip2pt
\noindent
This is no longer true for the fractional problem \eqref{dp}. For instance, the function  $u(x)=(1-|x|^{2})_{+}^{s}$ solves \eqref{dp} for $\Omega=B_{1}$, $p=2$ and $f= {\rm const.}$ in $\Omega$. Its interior regularity is $C^{\infty}$ (as the Schauder theory {\em a priori} forces for $C^{\infty}$ right-hand sides), but its boundary regularity is only $C^{s}$. Thus, we see that there is no obvious way to reproduce the odd reflection trick to deduce boundary regularity for \eqref{dp}, since actually boundary and interior regularity are quantitatively different.
\vskip2pt
\noindent
The first result dealing with the  boundary regularity for problem \eqref{dp} is contained in \cite{RS} for $p=2$, where it is proved that $u\in C^{s}(\R^{N})$ whenever the non-homogeneous term is bounded. In the nonlinear case, \cite{IMS1, IMS2} contain a global H\"older continuity result, with an unspecified H\"older exponent (see also \cite{KKP} for a refinement and generalisation when $f= 0$). Coupling the barrier argument contained in \cite{IMS1} with the optimal interior regularity of \cite{BLS} provides the optimal regularity $u\in C^{s}(\R^{N})$ when $p\ge 2$. The same is expected to be true in the case $p\in (1, 2)$, but the optimal (at least $C^{s}$) interior regularity in this framework is missing. 
\vskip2pt
\noindent
Still, even in the linear case, there is much more to be said. Despite the optimal regularity $u\in C^{s}(\overline\Omega)$ rules out in general the existence of the classical normal derivative, in the seminal paper \cite{RS} a regularity result for the $s$-normal derivative
\[\frac{\partial u}{\partial\nu^{s}}(x_{0}):=\lim_{t\to 0^+}\frac{u(x_{0}+t\nu_{x_{0}})}{t^{s}},\]
where $\nu_{x_{0}}$ denotes the inner normal to $\partial\Omega$ at $x_{0}\in \partial\Omega$. More precisely, they proved that, if when $p=2$ and $\partial\Omega$ is $C^{1,1}$, then any solution $u$ of \eqref{dp} satisfies
\[\left\|\frac{u}{\ds}\right\|_{C^{\alpha}(\overline\Omega)}\leq C\|f\|_{L^{\infty}(\Omega)},\quad {\rm d}_\Omega(x):={\rm dist}(x, \partial\Omega)\]
for some $\alpha=\alpha(N, s, \Omega)\in (0, 1)$, $C=C(N, s, \Omega)>0$.
\vskip2pt
\noindent
The latter can also be seen as a weighted H\"older regularity result and it provided several  applications to ovedetermined problems \cite{FJ}, nonlinear analysis \cite{IMS, ILPS}, free boundary problems \cite{CRS} and integration by parts formula \cite{RSV}. For further references and related results we refer to the survey article \cite{RO}.
\vskip2pt
\noindent
Our main contribution is an analogous fine boundary regularity result for the weak solution to \eqref{dp} in the degenerate case $p\ge 2$.

\begin{theorem}\label{main}
Let $p\ge 2$, $\Omega$ be a bounded domain with $C^{1,1}$ boundary and ${\rm d}_\Omega(x):={\rm dist}(x, \partial\Omega)$. Then there exist $\alpha\in (0,s]$ and $C>0$, depending on $N$, $p$, $s$, and $\Omega$, s.t.\ for all $f\in L^\infty(\Omega)$ the weak solution $u\in W^{s,p}_0(\Omega)$ to problem \eqref{dp} satisfies $u/\ds\in C^\alpha(\overline\Omega)$ and
\[\Big\|\frac{u}{\ds}\Big\|_{C^\alpha(\overline\Omega)}\le C\|f\|_{L^\infty(\Omega)}^\frac{1}{p-1}.\]
\end{theorem}

\noindent
With the result above we hope to provide nonlocal regularity theory with an analog of Lieberman's $C^1(\overline\Omega)$ regularity theorem for the (local) $p$-Laplacian \cite{L}. We privilege weak solutions (e.g., with respect to viscosity solutions, see \cite{L1}) mainly because we consider problem \eqref{dp} in a variational perspective. One possible development of this research is towards a nonlinear extension of the main result of \cite{IMS}, i.e., the equivalence of Sobolev and weighted H\"older local minimizers for the energy functional of a nonlinear boundary value problem driven by $\fpl$, which will require a regularity result like Theorem \ref{main} above for a slightly more general nonlocal, nonlinear operator modeled on $\fpl$ (as in \cite{GPM} for the local case $s=1$). The singular case $p\in (1,2)$ of Theorem \ref{main} remains open, but it can probably be dealt with through suitable variations of the techniques presented here.
\vskip2pt
\noindent
{\bf Sketch of proof.} Our aim is a weak Harnack inequality for the function $u/\ds$, and in particular a pointwise control of  $u/\ds$ in terms of an integral quantity. Our strategy is to exploit the nonlocality of the operator and define the following nonlocal excess
\[\Ex(u, k, R, x_{0})=\dashint _{\tilde B_{R, x_{0}}}\Big|\frac{u}{\ds}-k\Big|\, dx\]
with $k\in\R$, $R>0$, and $\tilde B_{R, x_{0}}$ being a small ball of radius comparable to $R$, placed at distance greater than $R$ in the inner normal direction from $x_{0}\in \partial\Omega$ (see figure \ref{fig1} and properties \eqref{bt} for a precise definition).  We call it nonlocal because it turns out that, given a bound on $\fpl u$, the pointwise behaviour of $u/\ds$ {\em inside $B_{R}(x_{0})\cap \Omega$} is controlled by the magnitude of the excess of $u$ in $\tilde B_{R, x_{0}}$, which takes into account the behaviour of $u/\ds$ {\em outside of $B_{R}(x_{0})\cap \Omega$}. 
\vskip2pt
\noindent
In order to describe the scheme of the proof, consider the case of $\Omega$ being the half-space $\R^{N}_{+}=\{x_{N}>0\}$, $x_{0}=0$, $R=1$, and $D_{1}=B_{1}\cap \R^{N}_{+}$.  We are going to prove two types of weak Harnack inequalities. The first one is for supersolutions and reads
\beq
\label{intlow}
\begin{cases}
\fpl u\ge 0 &\text{in $D_1$}\\
u\ge \ds&\text{in $\R^{N}_{+}$}
\end{cases}
\quad \Longrightarrow\quad \inf_{B_{1/4}\cap \R^{N}_{+}}\Big(\frac{u}{\ds}-1\Big)\ge \sigma\,\Ex(u).
\eeq
Here $e_{N}=(0, \dots, 1)$, $B_{1/4}$ is centered at $0$ and $\sigma$ is a positive constant depending only on $N, p$, and $s$. Besides, the translated ball $e_{N}+B_{1/4}$ corresponds to $\tilde B_{1}$ and we have set
\[\Ex(u)=\Ex(u,1,1,0)=\dashint_{e_{N}+B_{1/4}}\Big(\frac{u}{\ds}-1\Big)\, dx.\]
The second one regards subsolutions and is
\beq
\label{intup}
\begin{cases}
\fpl u\le 0 &\text{in $D_1$}\\
u\le \ds&\text{in $\R^{N}_{+}$}
\end{cases}
\quad \Longrightarrow\quad \inf_{B_{1/4}\cap \R^{N}_{+}}\Big(1-\frac{u}{\ds}\Big)\ge \sigma\,  \Ex(u).
\eeq
Note that in both cases we have a precise sign information on the difference $u/\ds-1$ in the translated ball. The similarity of the two statements is misleading, since, as will be seen later, the second one is actually considerably more difficult to prove than the first one. 
\vskip2pt
\noindent
The reason why these kind of nonlocal weak Harnack inequality hold lies in the following nonlocal superposition principle, which in a different form was proved in \cite{IMS1}. Given a regular function $w$ and a perturbation $u$, define 
\[\widetilde{w}_u= w+(u-w)\chi_{\tilde B_{1}}\]
Then, under some mild control of $w$ in terms of $\ds$ on $\tilde{B}_{1}$, we have
\beq
\label{intnls}
\begin{cases}
u\ge w \text{ in $\tilde B_{1}$}\quad \Longrightarrow\quad (-\Delta)^s \widetilde{w}\le (-\Delta)^sw -c\, \Ex(u)\quad \text{in $D_1$}\\[2pt]
u\le w \text{ in $\tilde B_{1}$}\quad \Longrightarrow\quad (-\Delta)^s \widetilde{w}\ge (-\Delta)^sw +c\, \Ex(u)\quad \text{in $D_1$}
\end{cases}
\eeq
for some $c=c(N, p, s)>0$.
\vskip2pt
\noindent
Our strategy for proving, e.g., \eqref{intlow} can then be roughly described as follows:
\begin{enumerate}
\item
Build a one parameter family of basic barrier $w_{\lambda}$ ($\lambda\in\R$) obeying the bounds
\beq
\label{intcontrols}
\begin{cases}
|\fpl w_{\lambda}|\le C\lambda&\text{in $\tilde B_{1}$}\\
w_{\lambda}\ge  (1+\lambda)\ds&\text{in $D_{1/4}$}\\
w_{\lambda}\le \ds &\text{in $D_{1}^{c}$}
\end{cases}
\eeq
\item
Choose $\lambda\simeq \Ex(u)$ so that the nonlocal superposition principle \eqref{intnls} ensures 
\[\fpl \widetilde w_{\lambda}\le 0\le \fpl u\quad \text{in $D_{1}$}.\]
and thanks to the global control $w_{\lambda}\le u$ in $D_{1}^{c}$, deduce that $\widetilde w_{\lambda}$ is an actual lower barrier for $u$. Thus, by comparison, $w=w_{\lambda}\le u$ in $D_{1/4}$.
\item
Conclude from the second condition in \eqref{intcontrols} that 
\[\frac{u}{\ds}-1\ge \frac{w}{\ds}-1\ge \lambda\simeq \Ex(u)\quad \text{in $D_{1/4}$}.\]
\end{enumerate}
\vskip2pt
\noindent
Most of the paper will thus be devoted to the construction of the family of basic barriers satisfying \eqref{intcontrols}. As it turns out, the construction will depend on the size of $\Ex(u)$, and we will need three different kinds of barriers. More precisely, for small values of $\Ex(u)$ (and thus of $\lambda$), we will build the barrier $w_{\lambda}$ starting from $\ds$ (which in the case of a half-space obeys $\fpl\ds=0$) and performing  a $C^{1,1}$-small diffeomorphism of the domain supported in $D_{1}$, to get the first condition in \eqref{intcontrols}.
A similar construction yields the upper barrier to prove \eqref{intup} in the case of small excess.
\vskip2pt
\noindent
For large values of $\Ex(u)$, the lower barrier will be a multiple (of order $\simeq\lambda$) of the torsion function
\[\begin{cases}
\fpl v=1 &\text{in $D_{1/2}$}\\
v\equiv 0&\text{in $D_{1/2}^{c}$},
\end{cases}\]
which, thanks to a Hopf type lemma and the size of $\Ex(u)\simeq\lambda$, fulfills the second bound in \eqref{intcontrols}.
\vskip2pt
\noindent
Unfortunately, when we are looking for the corresponding basic {\em upper} barrier $w_{\lambda}$ for large $\Ex(u)\simeq\lambda$, namely
\[
\begin{cases}
|\fpl w_{\lambda}|\le C\lambda&\text{in $\tilde B_{1}$}\\
w_{\lambda}\le  (1-\lambda)\ds&\text{in $D_{1/4}$}\\
w_{\lambda}\ge \ds &\text{in $D_{1}^{c}$}
\end{cases}
\]
(in order to prove the weak Harnack inequality for subsolutions \eqref{intup}), the previous construction fails. Indeed, when $\lambda>1$, $w_{\lambda}/\ds$ must change sign near $\partial\Omega\cap (D_{1}\setminus D_{1/4})$ and, even in the case of a half-space, we lack explicit examples of functions with bounded $\fpl$ having such behaviour. To get around this difficulty we employ an abstract construction chiefly based on the Lewy-Stampacchia inequality, building an upper barrier which solves a double obstacle problem. This ensures that, for large excess, the solution $u$ is nonpositive in $D_{1/2}$, and now the torsion function argument applies providing the desired bounds. 
\vskip2pt
\noindent
Finally, we localize  \eqref{intlow} and \eqref{intup}, requiring the pointwise bounds to hold only in $D_{2}$. This is done by looking at the truncations of $u$ below or above $\ds$ and, due to the nonlocality of the operator, it produces additional non-homogeneous terms (usually called {\em tails} in the literature) which in the case $p\ge 2$ are quite delicate to care of (see Remark \ref{rem00} in this respect). Having the local version of the weak Harnack inequality finally gives the desired H\"older continuity through well known techniques, originally developed in \cite{RS} for the linear case.
\vskip2pt
\noindent
{\bf Notation.} Throughout the paper, dependence on $N$, $p$, $s$ will often be omitted. Positive constants will be denoted by $C_1,C_2,\ldots$ When measurable functions are involved, the expression 'in $\Omega$' will always mean 'a.e.\ in $\Omega$' (and similar). We will regularly set $a^{p-1}=|a|^{p-2}a$ for all $a\in\R$. The positive order cone of a function space $X$ is denoted $X_+$. For all function $f$, we denote by $f_+$ its positive part. Functions defined in a domain $U\subset\R^N$ will be identified with their extensions to $\R^N$ vanishing in $U^c$. The minimum (resp.\ maximum) of two functions $f$, $g$ is denoted by $f\wedge g$ (resp.\ $f\vee g$). Though our main theorem is only proved for $p\ge 2$, all the intermediate results will, unless otherwise stated, hold for any $p>1$.

\section{Preliminaries}\label{sec2}

\noindent
As we said in Section \ref{sec1}, $\Omega\subset\R^N$ will always be a bounded domain with a $C^{1,1}$ boundary $\partial\Omega$.  For all $x\in\R^N$ and $R>0$ we set
\[B_R(x)=\big\{y\in\R^N:\,|x-y|<R\big\}, \quad D_R(x)=B_R(x)\cap\Omega\]
(we omit the $x$-dependence if $x=0$, i.e., we set $B_R(0)=B_R$, $D_R(0)=D_R$). We define a distance function by setting for all $x\in\R^N$
\[{\rm d}_\Omega(x)=\inf_{y\in\Omega^c}\,|x-y|.\]
Clearly ${\rm d}_\Omega:\R^N\to\R_+$ is $1$-Lipschitz continuous. By the $C^{1,1}$-regularity of $\partial\Omega$, $\Omega$ has the {\em interior sphere property}, namely there exists $R>0$ s.t.\ for all $x\in\partial\Omega$ we can find $y\in\Omega$ s.t.\ $B_{2R}(y)\subseteq\Omega$ is tangent to $\partial\Omega$ at $x$ (in some result this weaker property alone will suffice). We denote by $\rho>0$ the supremum of such $R$'s i.e.
\beq
\label{grho}
\rho=\rho(\Omega)=\sup\{R: \forall\, x\in \partial\Omega \ \exists\,  B_{2R}\subseteq \Omega \text{ s.t. } x\in \partial B_{2R}\}>0
\eeq
and define the neighborhood of  $\partial\Omega$ by setting
\[\Omega_\rho=\big\{x\in\Omega:\,{\rm d}_\Omega(x)<\rho\big\}.\]
By the choice of $\rho$, the metric projection $\Pi_\Omega:\Omega_\rho\to\partial\Omega$ is well defined and is $C^{1,1}$ if $\partial\Omega$ is $C^{1,1}$. Moreover, (see figure \ref{fig1}) for all $x\in\partial\Omega$ and $R\in\ ]0,\rho[$ there exists a ball $\tilde B_{x,R}$ of radius $R/4$ s.t.
\beq\label{bt}
\tilde B_{x,R}\subset D_{2R}(x)\setminus D_{3R/2}(x), \quad \inf_{y\in\tilde B_{x,R}}{\rm d}_\Omega(y)\ge\frac{3R}{2}.
\eeq

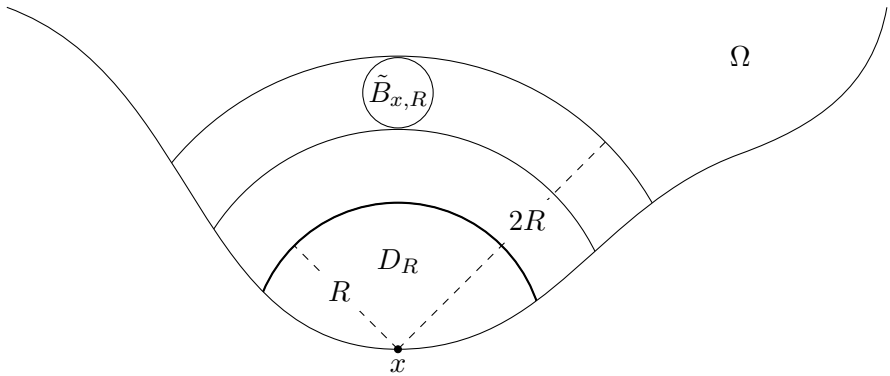
\begin{figure}
\centering
\begin{tikzpicture}[scale=1.3]
\filldraw (0, 0) circle (1pt);
\draw (0, 0) node[below]{$x$};
\draw[clip] (-4,3.5) to [out=-20, in=180] (0, 0) to [out=0, in=200] (3.5,2) to [out=20, in=-100] (5, 3.5);
\draw[thin] (0,0) circle (2.25);
\draw (3.5,3) node{$\Omega$};
\draw (0, 2.625) node{$\tilde{B}_{x, R}$};
\draw (0, 2.625) circle (0.36);
\draw (0, 0.9) node{$D_{R}$};
\draw[clip] (0, 0) circle (3);
\draw[dashed] (0, 0) -- (4,4);
\draw (1.32,1.32) node[fill=white]{$2R$};
\draw[thick] (0,0) circle (1.5);
\draw[clip] (0,0) circle (1.5);
\draw[dashed] (0, 0) -- (-2, 2);
\draw (-0.6, 0.6) node[fill=white]{$R$};
\end{tikzpicture}
\caption{The ball $\tilde{B}_{x, R}$, with center in the normal direction.}
\label{fig1}
\end{figure}

\noindent
We recall now the definitions of the main function spaces that we shall use in this paper. For all measurable $u:\R^N\to\R$ we set
\[[u]_{s,p}^p=\iint_{\R^N\times\R^N}\frac{|u(x)-u(y)|^p}{|x-y|^{N+ps}}\,dx\,dy,\]
and we define the fractional Sobolev space
\[W^{s,p}(\R^N)=\big\{u\in L^p(\R^N):\,[u]_{s,p}<\infty\big\},\]
which is a Banach space with respect to the norm $\|u\|_{s,p}=[u]_{s,p}+\|u\|_{L^{p}(\R^{N})}$, with $C^{\infty}_{c}(\R^{N})$ as a dense subspace. We also set
\[\w=\big\{u\in W^{s,p}(\R^N):\,u=0 \quad \text{in $\Omega^c$}\big\}\]
(equivalent to the definition given in Section \ref{sec1}), the latter being a uniformly convex, separable Banach space with the norm $[u]_{s,p}$. The dual space of $\w$ is denoted by $W^{-s,p'}(\Omega)$. We will also use the following function space:
\[\widetilde{W}^{s,p}(\Omega)=\Big\{u\in L^p_{\rm loc}(\R^N):\,\text{ $\exists\ \Omega'\Supset\Omega$ s.t.\ $u\in W^{s,p}(\Omega')$ and} \ \int_{\R^N}\frac{|u(x)|^{p-1}}{(1+|x|)^{N+ps}}\,dx<\infty\Big\}.\]
Such space plays an important r\^ole in the study of our problem, since by \cite[Lemma 2.3]{IMS1} for all $u\in\widetilde{W}^{s,p}(\Omega)$ we have $\fpl u\in W^{-s,p'}(\Omega)$. We also set, for any open subset $U\subset\Omega$,
\[\widetilde{W}^{s,p}_0(U)=\big\{u\in\widetilde{W}^{s,p}(U): \ u(x)=0 \ \text{in $\Omega^c$}\big\}\]
(note that $u$ does not necessarily vanish in all of $U^c$). We define a notion of {\em nonlocal tail} (slightly different from that introduced in \cite{DKP}) by setting for all measurable $u:\R^N\to\R$, $R>0$, and $q\ge 1$
\beq\label{tail}
{\rm tail}_q(u,R)=\Big[\int_{\Omega\cap B^c_R}\frac{|u(x)|^q}{|x|^{N+s}}\,dx\Big]^\frac{1}{q}.
\eeq
All equations and inequalities involving $\fpl$ are meant in the {\em weak} sense, unless explicitly stated: e.g., for any $u\in\w$ and $f\in L^\infty(\Omega)$, we say that $\fpl u=f$ in $\Omega$, if 
\[\iint_{\R^N\times\R^N}\frac{(u(x)-u(y))^{p-1}(\varphi(x)-\varphi(y))}{|x-y|^{N+ps}}\,dx\,dy=\int_\Omega f(x)\varphi(x)\,dx\qquad \text{for all $\varphi\in\w$}.\]
Similarly, we say that $\fpl u\le f$ in $\Omega$ if for all $\varphi\in\w_+$
\[\iint_{\R^N\times\R^N}\frac{(u(x)-u(y))^{p-1}(\varphi(x)-\varphi(y))}{|x-y|^{N+ps}}\,dx\,dy\le\int_\Omega f(x)\varphi(x)\,dx.\]
\vskip2pt
\noindent
We will also use the space of $\alpha$-H\"older continuous functions
\[C^\alpha(\overline\Omega)=\big\{u\in C(\overline\Omega):\,[u]_{C^\alpha(\overline\Omega)}<\infty\big\},\]
where
\[[u]_{C^\alpha(\overline\Omega)}=\sup_{x,y\in\overline\Omega,\,x\neq y}\frac{|u(x)-u(y)|}{|x-y|^\alpha},\]
which is a Banach space endowed with the norm
\[\|u\|_{C^\alpha(\overline\Omega)}=\|u\|_{L^\infty(\Omega)}+[u]_{C^\alpha(\overline\Omega)}.\]
In the rest of the section we will list some useful technical results on solutions to \eqref{dp} type problems on several domains: {\em for simplicity, we always denote the domain by $\Omega$, but in the forthcoming sections these results will also be applied to different domains.}
\vskip2pt
\noindent
We begin with the following weak comparison principle (see \cite[Lemma 9]{LL}, \cite[Proposition 2.10]{IMS1}):

\begin{proposition}\label{comp}
{\rm (Comparison principle)} Let $u,v\in\widetilde{W}^{s,p}(\Omega)$ satisfy
\[\begin{cases}
\fpl u\le\fpl v & \text{in $\Omega$} \\
u\le v & \text{in $\Omega^c$.}
\end{cases}\]
Then $u\le v$ in $\R^N$.
\end{proposition}

\noindent
Our first result is a simple estimate on the solution to the torsion equation in a ball: for all $R>0$, we denote by $u_R\in W^{s,p}_0(B_R)$ the (unique) solution to
\beq\label{torball}
\begin{cases}
\fpl u_R=1 & \text{in $B_R$} \\
u_R=0 & \text{in $B_R^c$.}
\end{cases}
\eeq

\begin{lemma}\label{torest}
There exists $C_1=C_1(N,p,s)>1$ s.t.\ for all $R>0$, $x\in\R^N$
\[\frac{R^\frac{s}{p-1}}{C_1}{\rm d}^s_{B_R}(x)\le u_R(x)\le C_1R^\frac{s}{p-1}{\rm d}^s_{B_R}(x).\]
\end{lemma}
\begin{proof}
First assume $R=1$. By the strong maximum principle (see \cite[Lemma 2.3]{MS}), we have $u_1>0$ in $B_1$, while by \cite[Theorem 4.4]{IMS1} there exists $C>0$ s.t.
\beq\label{torest1}
u_1\le C{\rm d}_{B_1}^s \quad \text{in $\R^N$}.
\eeq
By \cite[Theorem 3.6]{IMS1} we can find $r\in\ ]0,1[$, $M>0$ s.t.\ $|\fpl{\rm d}_{B_1}^s|\le M$ in $B_1\setminus\overline B_r$. Set $m=\inf_{B_r}u_1>0$ and for all $x\in\R^N$
\[w(x)=\min\big\{m,M^{-\frac{1}{p-1}}\big\}{\rm d}_{B_1}^s(x).\]
Then we have
\[\begin{cases}
\fpl w\le\fpl u_1 & \text{in $B_1\setminus\overline B_r$} \\
w\le u_1 & \text{in $(B_1\setminus\overline B_r)^c$.}
\end{cases}\]
Proposition \ref{comp} yields $w\le u_1$ in $\R^N$. So, for $C_1$ even bigger if necessary in \eqref{torest1}, we improve to
\beq\label{torest2}
\frac{{\rm d}_{B_1}^s}{C_1}\le u_1\le C_1{\rm d}_{B_1}^s \quad \text{in $\R^N$}.
\eeq
Now take an arbitrary $R>0$ and set for all $x\in\R^N$
\[v(x)=\frac{u_R(Rx)}{R^{p's}}.\]
Then $v\in W^{s,p}_0(B_1)$ and by the homogeneity and scaling properties of $\fpl$ (see \cite[Proposition 2.9 $(i)$ $(ii)$]{IMS1}) we have
\[\begin{cases}
\fpl v=1 & \text{in $B_1$} \\
v=0 & \text{in $B_1^c$.}
\end{cases}\]
By uniqueness $v=u_1$. Since ${\rm d}_{B_R}(Rx)=R{\rm d}_{B_1}(x)$, by \eqref{torest2} we have for all $x\in\R^N$
\[\frac{{\rm d}_{B_R}^s(Rx)}{C_1R^s}\le\frac{u_R(Rx)}{R^{p's}}\le\frac{C_1{\rm d}_{B_R}^s(Rx)}{R^s},\]
hence the conclusion. 
\end{proof}

\noindent
The previous estimate allows us to use $u_R$ as a barrier to prove a Hopf type lemma for the torsion equation in a general domain:
\beq\label{torsion}
\begin{cases}
\fpl u=1 & \text{in $\Omega$} \\
u=0 & \text{in $\Omega^c$}.
\end{cases}
\eeq

\begin{lemma}\label{hopf}
{\rm (Hopf's lemma)} Let $u\in\w$ solve \eqref{torsion} and $\Omega$ satisfy the interior sphere property \eqref{grho}. Then  
\[u(x)\ge C_1\rho^\frac{s}{p-1}\ds(x)\qquad \text{for all $x\in\R^N$},\]
where  $C_{1}=C_{1}(N, p, s)>1$ is given in the previous Lemma.
\end{lemma}
\begin{proof}
First, fix $x\in\Omega_\rho$. Then we can find a ball $B\subseteq\Omega$ of radius $2\rho$, tangent to $\partial\Omega$ at $\Pi_\Omega(x)$ and such that ${\rm d}_{\Omega}(x)={\rm d}_{B}(x)$. Let $v\in W^{s,p}_0(B)$ solve
\[\begin{cases}
\fpl v=1 & \text{in $B$} \\
v=0 & \text{in $B^c$.}
\end{cases}\]
So we have
\[\begin{cases}
\fpl v\le\fpl u & \text{in $B$} \\
v\le u & \text{in $B^c$.}
\end{cases}\]
By Proposition \ref{comp} we have $v\le u$ in $\R^N$. By Lemma \ref{torest} and  ${\rm d}_\Omega(x)={\rm d}_B(x)$, we infer
\beq\label{hopf1}
u(x)\ge v(x)\ge\frac{(2\rho)^\frac{s}{p-1}}{C_1}\ds(x).
\eeq
Now assume $x\in\Omega\setminus\overline\Omega_\rho$, and set $R={\rm d}_\Omega(x)\ge\rho$. The ball $B'=B_{R}(x)$ is contianed in $\Omega$ and  ${\rm d}_{B'}(x)=R={\rm d}_{\Omega}(x)$. Considering the torsion function $v'$ of $B'$ and applying Proposition \ref{comp}, we deduce through Lemma \ref{torest}
\beq\label{hopf2}
u(x)\ge v'(x)\ge \frac{R^\frac{s}{p-1}}{C_1}{\rm d}_{B'}^s(x)=\frac{R^\frac{s}{p-1}}{C_1}R^{s}\ge\frac{\rho^\frac{s}{p-1}}{C_1} \ds.
\eeq
From \eqref{hopf1} and \eqref{hopf2} we conclude.
\end{proof}

\noindent
Another property of problem \eqref{torsion} is that its solution is a subsolution all over $\R^N$:

\begin{lemma}\label{subglobal}
Let $\Omega\subseteq\R^{N}$ be bounded and $u\in\w$ solve \eqref{torsion}. Then $\fpl u\le 1$ in $\R^N$.
\end{lemma}
\begin{proof}
Set for all $v\in W^{s,p}(\R^N)\cap L^1(\R^N)$
\[J_1(v)=\frac{[v]_{s,p}^p}{p}-\int_{\R^N} v(x)\,dx,\]
and consider the minimization problem
\[\min_{v\le u}J_1(v).\]
By strict convexity and coercivity, it admits a unique solution $v_0\in W^{s,p}(\R^N)\cap L^1(\R^N)$. Since $J_1(v_+)\le J_1(v)$ for all admissible $v$, we have $v_0\ge 0$ in $\R^N$. Moreover, $v_0\le u=0$ in $\Omega^c$, so $v_0\in\w$ and it is readily seen that $v_{0}$ solves \eqref{torsion}. Recalling that $u$ is the only weak solution of \eqref{torsion}, we conclude that $v_0=u$. From the minimality property of $v_{0}$, we deduce
\[\langle J'_{1}(v_{0}), v-u\rangle\ge 0\quad \text{for all $v\le u$, $v\in W^{s,p}(\R^N)\cap L^1(\R^N)$},\]
and setting $v=u-\varphi$, we get
\[\langle J'_1(u),\varphi\rangle\le 0\quad \text{for all $\varphi\in C^{\infty}_{c}(\R^{N})$, $\varphi\ge 0$},\]
i.e., $\fpl u\le 1$ in all of $\R^N$.
\end{proof}

\noindent
We introduce a partial ordering on the dual space $W^{-s,p'}(\Omega)$ by defining the positive cone
\[W^{-s,p'}(\Omega)_+=\big\{L\in W^{-s,p'}(\Omega):\,\langle L,\varphi\rangle\ge 0 \ \text{for all $\varphi\in\w_+$}\big\}.\]
By the Riesz theorem and the density of $C^\infty_c(\Omega)$ in $\w$, any $L\in W^{-s,p'}(\Omega)_+$ can be faithfully represented as a (positive) Radon measure on $\Omega$ (see the discussion in \cite[p.\ 265]{GM}). Then, the {\em order dual} of $w$ is defined as
\[W^{-s,p'}_\le(\Omega)=\big\{L_1-L_2:\, L_1,L_2\in W^{-s,p'}(\Omega)_+\big\}.\]
Such space inherits a lattice structure defined by duality through the lattice structure of $\w$, as shown in \cite[p.\ 260]{GM}. We now give a slight generalization of the Lewy-Stampacchia type inequality \cite[Theorem 2.4]{GM} which is needed to treat double obstacle problems with obstacle not lying in $W^{s,p}_{0}(\Omega)$. The proof is well known and we describe it for sake of completeness, specializing to the case of the operator $\fpl$.

\begin{lemma}\label{dobstacle}
{\rm (Lewy-Stampacchia)} Let $\Omega\subseteq \R^{N}$ be bounded, $\varphi, \psi\in W^{s,p}_{\rm loc}(\R^{N})$ be s.t.\ 
\begin{enumroman}
\item
\label{order1}
$\fpl\varphi, \fpl \psi\in W^{-s, p'}_{\le}(\Omega)$ 
\item
\label{order2}
$[\varphi, \psi]:=\big\{v\in W^{s,p}_{0}(\Omega): \varphi\le v\le \psi\big\}\neq \emptyset$
\end{enumroman}
Then there exists a unique solution $u\in W^{s, p}_{0}(\Omega)$ to the problem
\[\min_{v\in [\varphi, \psi]}\frac{[v]^{p}_{s,p}}{p},\]
and it satisfies
\[0\wedge\fpl\psi\le\fpl u\le 0\vee\fpl\varphi \quad \text{in $\Omega$}.\]
\end{lemma}
\begin{proof}
The existence and uniqueness statements for the minimization problem follow from the strict convexity and coercivity of $v\mapsto [v]_{s,p}^{p}$. The function $u\in[\varphi,\psi]$ is a minimizer iff it satisfies for all $v\in [\varphi, \psi]$
\beq
\label{minu}
\langle \fpl u, v-u\rangle\ge 0.
\eeq
We prove now that
\beq\label{lsin}
\fpl u\le 0\vee\fpl\varphi \quad \text{in $\Omega$.}
\eeq
Recall from \cite[Remark 3.3 and p.\ 261]{GM} that $v\mapsto [v]_{s,p}^{p}/p$ is sub-modular and strictly convex, hence its differential $\fpl$ is a strictly ${\mathcal T}$-monotone map, i.e.
\beq
\label{tmon}
\langle \fpl u-\fpl v, (u-v)_{+}\rangle >0 \quad \text{unless $v\le u$}.
\eeq
By condition \ref{order1}, the strictly convex, coercive functional  
\[J_2:W^{s, p}_{0}(\Omega)\to \R,\quad J_2(v)=\frac{[v]_{s,p}^{p}}{p}- \langle\fpl \varphi\vee 0, v\rangle,\]
is well defined, and we thus let $w$ be the unique solution of the following problem
\[\min_{v\in\,  (\infty, u]}J_2(v),\quad (-\infty, u]:= \big\{ v\in W^{s,p}_{0}(\Omega): v\le u\big\},\]
which therefore solves for all $v\in (-\infty, u]$
\beq
\label{minw}
\langle J_2'(w), v-w\rangle\ge 0.
\eeq
We claim that $u\ge w$, and then necessarily $u=w$. Condition \ref{order2} forces $\varphi\le 0$ in $\Omega^{c}$, therefore $w\vee \varphi\in W^{s,p}_{0}(\Omega)$. Choosing $v=w\vee\varphi=w+(\varphi-w)_{+}$  gives
\[
0 \le \langle J_2'(w), (\varphi-w)_{+}\rangle = \langle \fpl w-(0\vee\fpl\varphi), (\varphi-w)_{+}\rangle \le \langle \fpl w-\fpl\varphi, (\varphi-w)_{+}\rangle.
\]
By \eqref{tmon}, this implies $\varphi\le w$ and, by $w\le u$, {\em a fortiori} $w\in [\varphi,\psi]$. Choosing $v=w\lor u=w+(u-w)_{+}$ as a test function in \eqref{minw} gives 
\[
0 \le \langle J_2'(w), (u-w)_{+}\rangle = \langle \fpl w-(0\vee\fpl\varphi),(u-w)_{+}\rangle \le \langle \fpl w, (u-w)_{+}\rangle, 
\]
while letting $v=w\land u=u-(u-w)_{+}$ in \eqref{minu}, provides
\[0\le \langle \fpl u, - (u-w)_{+}\rangle.\]
Summing up we obtain
\[0\le \langle \fpl w-\fpl u, (u-w)_{+}\rangle,\]
thus \eqref{tmon} entails $u\le w$ and therefore $w=u$. This enforces \eqref{minw} for $u$, then setting $v=u-z\in \, (-\infty, u]$ we get for all $z\in W^{s,p}_{0}(\Omega)_{+}$
\[\langle \fpl u, z\rangle\le \langle \fpl\varphi\lor 0, z\rangle,\]
proving \eqref{lsin}. The first inequality of the thesis is achieved through a similar argument.
\end{proof}

\noindent
A major tool in our proofs is the following nonlocal superposition principle:

\begin{proposition}\label{spp}
{\rm (Superposition principle)} Let $\Omega$ be bounded, $u\in\widetilde{W}^{s,p}(\Omega)$, $v\in L^1_{\rm loc}(\R^N)$, $V={\rm supp}(u-v)$ satisfy
\begin{enumroman}
\item $\Omega\Subset \R^{N}\setminus V$;
\item $\displaystyle\int_V\frac{|v(x)|^{p-1}}{(1+|x|)^{N+ps}}\,dx<\infty$.
\end{enumroman}
Set for all $x\in\R^N$
\[w(x)=\begin{cases}
u(x) & \text{if $x\in V^c$} \\
v(x) & \text{if $x\in V$.}
\end{cases}\]
Then $w\in\widetilde{W}^{s,p}(\Omega)$ and satisfies in $\Omega$
\[\fpl w(x)=\fpl u(x)+2\int_V\frac{(u(x)-v(y))^{p-1}-(u(x)-u(y))^{p-1}}{|x-y|^{N+ps}}\,dy.\]
\end{proposition}
\begin{proof}
We can rephrase $w=u+(v-u)\chi_V$, which implies $w\in\widetilde{W}^{s,p}(\Omega)$. By \cite[Lemmas 2.3, 2.8]{IMS1} we have $\fpl w\in W^{-s,p'}(\Omega)$, moreover for all $\varphi\in\w$
\[\langle\fpl w,\varphi\rangle=\langle\fpl u,\varphi\rangle+\int_\Omega h(x)\varphi(x)\,dx,\]
where for all Lebesgue point $x\in V$ of $u$ we have set
\[h(x)=2\int_V\frac{(u(x)-v(y))^{p-1}-(u(x)-u(y))^{p-1}}{|x-y|^{N+ps}}\,dy.\]
This concludes the proof.
\end{proof}

\noindent
We conclude this section with a key estimate  for a function which is {\em locally} bounded by a suitable multiple of $\ds$ (here we first require that $p\ge 2$). The passage from a global bound to a local bound can be delicate for a nonlocal operator such as $\fpl$. While technical, the next proposition shows the main reason why the degeneracy of the operator forces, in the following sections, a peculiar decomposition of the right hand side (see Remark \ref{rem00} below).

\begin{proposition}\label{updown}
Let $\Omega$ be bounded, $p\ge 2$ and $u\in\widetilde{W}^{s,p}_0(D_R)$ satisfy $\fpl u\in W^{-s,p'}_\le(D_R)$:
\begin{enumroman}
\item\label{updown1} if there exists $m\in\R$ s.t.\ $u\ge m\ds$ in $D_{2R}$, then for all $\eps>0$ there exist $C_\eps=C_\eps(N,p,s,\eps)>0$ and another constant $C_3=C_3(N,p,s)>0$ s.t.\ in $D_R$
\begin{align*}
\fpl\big(u\vee m\ds) &\ge \fpl u-\frac{\eps}{R^s}\Big\|\frac{u}{\ds}-m\Big\|_{L^\infty(D_R)}^{p-1}-C_\eps{\rm tail}_{p-1}\,\Big(\Big(m-\frac{u}{\ds}\Big)_+,2R\Big)^{p-1} \\
&\quad -C_3|m|^{p-2}{\rm tail}_1\,\Big(\Big(m-\frac{u}{\ds}\Big)_+,2R\Big);
\end{align*}
\item\label{updown2} if there exists $M\in\R$ s.t.\ $u\le M\ds$ in $D_{2R}$, then for all $\eps>0$ there exist $C'_\eps=C'_\eps(N,p,s,\eps)>0$ and another constant $C'_3=C'_3(N,p,s)>0$ s.t.\ in $D_R$
\begin{align*}
\fpl\big(u\wedge M\ds) &\le \fpl u+\frac{\eps}{R^s}\Big\|M-\frac{u}{\ds}\Big\|_{L^\infty(D_R)}^{p-1}+C'_\eps{\rm tail}_{p-1}\,\Big(\Big(\frac{u}{\ds}-M\Big)_+,2R\Big)^{p-1} \\
&\quad +C'_3|M|^{p-2}{\rm tail}_1\,\Big(\Big(\frac{u}{\ds}-M\Big)_+,2R\Big).
\end{align*}
\end{enumroman}
\end{proposition}
\begin{proof}
We prove \ref{updown1}. We may assume $u/\ds-m\in L^\infty(D_R)$, otherwise there is nothing to prove. We will use the following elementary inequality: since $p\ge 2$, there exists $C_p>0$ s.t.
\beq\label{abc}
(a-b)^{p-1}-(c-b)^{p-1}\le C_p(|a|^{p-2}+|b|^{p-2})|a-c|+C_p|a-c|^{p-1}\qquad \text{for all $a,b,c\in\R$}.
\eeq
Indeed, by Lagrange's theorem and convexity, we have
\begin{align*}
(a-b)^{p-1}-(c-b)^{p-1} &\le C_p(|a|^{p-2}+|b|^{p-2}+|c|^{p-2})|a-c| \\
&\le C_p\big(|a|^{p-2}+|b|^{p-2}+C'_{p}(|c-a|^{p-2}+|a|^{p-2})\big)|a-c|,
\end{align*}
which implies \eqref{abc}. Set $w=u\vee m\ds$. Since $\{u<m\ds\}\subseteq D_{2R}^c$ is bounded away from $D_R$, we can apply Proposition \ref{spp} and get for all $x\in D_R$
\beq\label{ud1}
\begin{split}
\fpl w(x) &= \fpl u(x)+2\int_{\{u<m\ds\}}\frac{(u(x)-m\ds(y))^{p-1}-(u(x)-u(y))^{p-1}}{|x-y|^{N+ps}}\,dy \\
&= \fpl u(x)-2\int_{\{u<m\ds\}}\frac{(m\ds(y)-u(x))^{p-1}-(u(y)-u(x))^{p-1}}{|x-y|^{N+ps}}\,dy.
\end{split}
\eeq
We use \eqref{abc} to estimate the numerator of the integrand, recalling also that ${\rm d}_\Omega(x)\le R$, $u(x)\ge m\ds(x)$, $R<{\rm d}_\Omega(y)\le |y|$, and $u(y)<m\ds(y)$:
\begin{align*}
 (m\ds(y)&-u(x))^{p-1}-(u(y)-u(x))^{p-1} \\
&\le C_p\big(|m\ds(y)|^{p-2}+|u(x)|^{p-2}\big)|m\ds(y)-u(y)|+C_p|m\ds(y)-u(y)|^{p-1} \\
&\le C_p\big(|m|^{p-2}{\rm d}_\Omega^{(p-2)s}(y)+|m|^{p-2}R^{(p-2)s}+(u(x)-m\ds(x))^{p-2}\big)(m\ds(y)-u(y))\\
&\quad+C_p(m\ds(y)-u(y))^{p-1} \\
&\le C|m|^{p-2}|y|^{(p-2)s}(m\ds(y)-u(y))+\eps(u(x)-m\ds(x))^{p-1}+C_\eps(m\ds(y)-u(y))^{p-1},
\end{align*}
where in the end we have also used Young's inequality with exponents $q=(p-1)(p-2)^{-1}$ and $q'=p-1$. Here $C>0$ depends only on $N$, $p$, $s$, while $C_\eps>0$ also depends on $\eps>0$. Now, by means of the inequality above and the relations $|x-y|\ge |y|/2 \ge R$, we can estimate the integral in \eqref{ud1}, getting
\begin{align*}
&\int_{\{u<m\ds\}}\frac{(m\ds(y)-u(x))^{p-1}-(u(y)-u(x))^{p-1}}{|x-y|^{N+ps}}\,dy \\
&\quad\le \eps\int_{\{u<m\ds\}}\frac{(u(x)-m\ds(x))^{p-1}}{|x-y|^{N+ps}}\,dy+C_\eps\int_{\{u<m\ds\}}\frac{(m\ds(y)-u(y))^{p-1}}{|x-y|^{N+ps}}\,dy \\
&\quad\quad+ C|m|^{p-2}\int_{\{u<m\ds\}}\frac{|y|^{(p-2)s}(m\ds(y)-u(y))}{|x-y|^{N+ps}}\,dy \\
&\quad\le \eps\Big\|m-\frac{u}{\ds}\Big\|_{L^\infty(D_R)}^{p-1}\int_{D_{2R}^c}\frac{R^{(p-2)s}}{|y|^{N+ps}}\,dy + C_\eps\int_{D_{2R}^c}\Big(m-\frac{u(y)}{\ds(y)}\Big)_+^{p-1}\,\frac{dy}{|y|^{N+s}} \\
&\quad\quad+ C|m|^{p-2}\int_{D_{2R}^c}\Big(m-\frac{u(y)}{\ds(y)}\Big)_+\,\frac{dy}{|y|^{N+s}} \\
&\quad\le \frac{\eps}{R^s}\Big\|\frac{u}{\ds}-m\Big\|_{L^\infty(D_R)}^{p-1}+C_\eps{\rm tail}_{p-1}\Big(\Big(m-\frac{u}{\ds}\Big)_+,\,2R\Big)^{p-1}+C_3{\rm tail}_1\Big(\Big(m-\frac{u}{\ds}\Big)_+,\,2R\Big),
\end{align*}
where we may take, if necessary, $\eps>0$ even smaller and $C_\eps>0$ even bigger, plus some $C_3(N, p, s)$. Plugging the last inequality into \eqref{ud1} (and replacing $\eps$ with $\eps/2$), we achieve \ref{updown1}.
\vskip2pt
\noindent
The argument for \ref{updown2} is immediate, by replacing $u$ with $-u$ and $m$ with $-M$.
\end{proof}

\begin{remark}
\label{rem00}
Before going further, a short discussion is in order. Proposition \ref{updown} provides bounds of the fractional $p$-Laplacians of truncated functions, which involve {\em two} tail terms with different exponents, namely ${\rm tail}_{p-1}$ and ${\rm tail}_1$. One of the main issues in the forthcoming sections will be to estimate inductively such tail terms, taking into account that they behave differently when $R\to 0^+$, with ${\rm tail}_{1}$ being asymptotically larger than ${\rm tail}_{p-1}$. In adjusting those estimates, the quantities $|m|^{p-2}$, $|M|^{p-2}$ multiplying the term ${\rm tail}_1$ in \ref{updown1}, \ref{updown2} respectively, will play a fundamental r\^ole. That is why we will emphasize the $m$-dependence of the right hand side for supersolutions (respectively, its $M$-dependence for subsolutions). Precisely, we shall prove a lower bound for a function $u$ satisfying
\[\begin{cases}
\fpl u\ge -K-m^{p-2}H & \text{in $D_R$} \\
u\ge m\ds & \text{in $\R^N$,}
\end{cases}\]
and an upper bound for a function $u$ satisfying
\[\begin{cases}
\fpl u\le K+M^{p-2}H & \text{in $D_R$} \\
u\le M\ds & \text{in $\R^N$,}
\end{cases}\]
respectively, with convenient $K,H,m,M>0$. As we will see, the upper and lower bounds require substantially different approaches.
\end{remark}

\section{The lower bound}\label{sec3}

\noindent
This section is devoted to the study of supersolutions of \eqref{dp} type problems, locally bounded from below by a multiple of $\ds$. For such supersolutions we aim at proving a lower bound for the quotient $u/\ds$ near the boundary (see Proposition \ref{lower} below).
\vskip2pt
\noindent
First, we assume that the supersolution $u$ is {\em globally} bounded from below by $m\ds$ and rephrase the lower bound on $\fpl u$ as $-K-m^{p-2}H$. Precisely, we assume $p\ge 2$, $0\in\partial\Omega$ (for simplicity of notation), $R\in\ ]0,\rho/4[$, and consider $u\in\widetilde{W}^{s,p}(D_R)$ satisfying for some $K,H,m\ge 0$
\beq\label{super}
\begin{cases}
\fpl u\ge -K-m^{p-2}H & \text{in $D_R$} \\
u\ge m\ds & \text{in $\R^N$.}
\end{cases}
\eeq
A major r\^ole in determining the behavior of $u/\ds$ in a semi-disc $D_{R}(x_{0})$ is played by the following nonlocal excess
\beq\label{lq}
\Ex(u, k, R, x_{0}) = \dashint_{\tilde B_R(x_{0})}\left|\frac{u(x)}{\ds(x)}-k\right|\,dx,
\eeq
where $\tilde B_R(x_{0})$ is defined as in \eqref{bt}. As we will frequently assume $x_{0}=0$, the dependence on the latter will be omitted. We begin by proving a lower bound for the case of large values of the excess, which highlights the nonlocal feature of the equation.

\begin{lemma}\label{lowerbig}
Let $u\in\widetilde{W}^{s,p}(D_R)$ solve \eqref{super}, $p\ge 2$ and $\Omega$ satisfy \eqref{grho}. Then there exist $\theta_1=\theta_1(N,p,s)\ge 1$, $C_4=C_4(N,p,s)>1$, $\sigma_1=\sigma_1(N,p,s)\in (0,1]$ s.t.\ for all $R\in \ ]0, \rho/4[$
\[\Ex(u, m, R)\ge m\theta_1 \ \Longrightarrow \ \inf_{D_{R/2}}\Big(\frac{u}{\ds}-m\Big) \ge \sigma_1\Ex (u, m, R)-C_4(KR^s)^\frac{1}{p-1}-C_4HR^s.\]
\end{lemma}
\begin{proof}
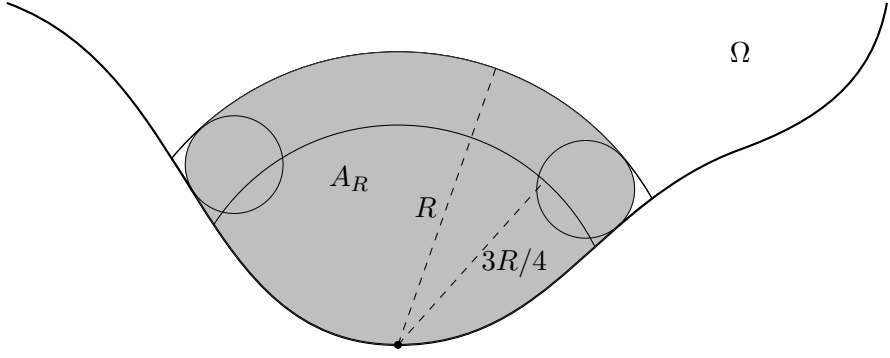
\begin{figure}
\centering
\begin{tikzpicture}[scale=1.3]
\filldraw[lightgray] (-1.675,1.845) circle (0.5);
\filldraw[lightgray] (1.92,1.592) circle (0.5);
\filldraw (0, 0) circle (1pt);
\draw (3.5,3) node{$\Omega$};
\draw[thick] (-4,3.5) to [out=-20, in=180] (0, 0) to [out=0, in=200] (3.5,2) to [out=20, in=-100] (5, 3.5);
\draw[clip] (-4,3.5) to [out=-20, in=180] (0, 0) to [out=0, in=200] (3.5,2) to [out=20, in=-100] (5, 3.5);
\draw (0, 0) circle (3);
\filldraw[lightgray] (0, 0) circle (2.56);
\draw[clip] (0, 0) circle (3);
\filldraw[lightgray]  (0, 0) -- (2.5, 2.05) -- (0, 4) -- (-2.26, 2.5) -- (0, 0);
\draw[very thin] (1.92,1.592) circle (0.5);
\draw[very thin] (-1.675,1.845) circle (0.5);
\draw (0, 0) circle (2.25);
\filldraw (0, 0) circle (1pt);
\draw (-0.5, 1.7) node{$A_{R}$};
\draw[dashed] (0,0) -- node[left]{$R$} (1, 2.83);
\draw[dashed] (0,0) -- node[right]{$3R/4$} (1.5,1.68);
\draw[thick] (-4,3.5) to [out=-20, in=180] (0, 0) to [out=0, in=200] (3.5,2) to [out=20, in=-100] (5, 3.5);
\draw (0, 0) circle (3);
\end{tikzpicture}
\caption{The regularized set $A_{R}$ in gray.}
\label{fig2}
\end{figure}
Set
\[A_R=\bigcup\Big\{B_r(y):\,y\in\R^N,\,r\ge\frac{R}{8},\,B_r(y)\subset D_R\Big\}.\]
By the regularity of $\partial\Omega$ stated in \eqref{grho} and $R<\rho/4$, $A_R\subset\R^N$ is a bounded domain satisfying the interior sphere condition with radius $\rho_{A_R}\ge R/16$ (see figure \ref{fig2}). Moreover we claim that 
\beq\label{lb1}
{\rm d}_\Omega\le C{\rm d}_{A_R} \quad \text{in $D_{R/2}$.}
\eeq
First note that $D_{3R/4}\subseteq A_R$ implies ${\rm d}_{D_{3R/4}}\le{\rm d}_{A_R}$ in $\R^N$. Furthermore, for all $x\in D_{R/2}$  we have
\[{\rm d}_\Omega(x) \le |x-\Pi_\Omega(\Pi_{D_{3R/4}}(x))| \le |x-\Pi_{D_{3R/4}}(x)|+|\Pi_{D_{3R/4}}(x)-\Pi_\Omega(\Pi_{D_{3R/4}}(x))|.\]
To proceed, we distinguish two cases:
\begin{itemize}[leftmargin=0.7cm]
\item[$(a)$] if $\Pi_{D_{3R/4}}(x)\in\partial\Omega$, then $\Pi_\Omega(\Pi_{D_{3R/4}}(x))=\Pi_{D_{3R/4}}(x)$ and so
\[{\rm d}_\Omega(x) \le {\rm d}_{D_{3R/4}}(x) \le {\rm d}_{A_R}(x);\]
\item[$(b)$] if $\Pi_{D_{3R/4}}(x)\notin\partial\Omega$, then we have $|\Pi_{D_{3R/4}}(x)|,\,|\Pi_\Omega(\Pi_{D_{3R/4}}(x))|\le R$ and ${\rm d}_{D_{3R/4}}(x)\ge R/4$, which in turn implies
\[{\rm d}_\Omega(x) \le {\rm d}_{D_{3R/4}}(x)+|\Pi_{D_{3R/4}}(x)|+|\Pi_\Omega(\Pi_{D_{3R/4}}(x))| \le 9{\rm d}_{D_{3R/4}}(x) \le 9{\rm d}_{A_R}(x).\]
\end{itemize}
Both cases lead to \eqref{lb1}. We will also use the following elementary inequality from \cite[Eq.\ (2.7)]{IMS1}:  since $p\ge 2$, for all $a\in\R$, $b\ge 0$ we have
\beq\label{ab}
(a+b)^{p-1}-a^{p-1}\ge 2^{2-p}b^{p-1}.
\eeq
Let $v\in W^{s,p}_0(A_R)$ be the solution of the torsion problem
\[\begin{cases}
\fpl v=1 & \text{in $A_R$} \\
v=0 & \text{in $A_R^c$.}
\end{cases}\]
By Lemma \ref{subglobal} we have $\fpl v\le 1$ in $\R^N$. Besides, by Lemma \ref{hopf} and \eqref{lb1} we have
\beq\label{lb2}
v\ge \frac{R^\frac{s}{p-1}}{C}\ds \quad \text{in $D_{R/2}$.}
\eeq
Pick $\lambda>0$ (to be  determined later) and set
\[w(x)=\begin{cases}
\displaystyle\frac{\lambda}{R^\frac{s}{p-1}}v(x) & \text{if $x\in\tilde B_R^c$} \\
u(x) & \text{if $x\in\tilde B_R$,}
\end{cases}\]
where $\tilde B_R$ is defined as in \eqref{bt}. We note that ${\rm dist}\,(\tilde B_R,\,D_R)>0$, so we can apply Proposition \ref{spp}. Also using homogeneity of $\fpl$, \eqref{ab}, and the relations ${\rm d}_\Omega(y)<3R/2$, $|x-y|>3R/4$, we get for all $x\in D_R$
\begin{align*}
\fpl w(x) &= \fpl \Big(\frac{\lambda}{R^\frac{s}{p-1}}v(x)\Big)+2\int_{\tilde B_R}\frac{(w(x)-u(y))^{p-1}-w^{p-1}(x)}{|x-y|^{N+ps}}\,dy \\
&\le \frac{\lambda^{p-1}}{R^s}-\frac{1}{C}\int_{\tilde B_R}\frac{u(y)^{p-1}}{|x-y|^{N+ps}}\,dy \\
&\le \frac{\lambda^{p-1}}{R^s}-\frac{1}{C}\int_{\tilde B_R}\frac{(u(y)-m\ds(y))^{p-1}}{|x-y|^{N+ps}}\,dy.
\end{align*}
Observe that by the property \eqref{bt} of $\tilde B_{R}$ 
\begin{align*}
\int_{\tilde B_R}\frac{(u(y)-m\ds(y))^{p-1}}{|x-y|^{N+ps}}\,dy &\ge \int_{\tilde B_R}\big(\frac{u(y)}{\ds(y)}-m\big)^{p-1}\frac{{\rm d}_\Omega^{s(p-1)}(y)}{|x-y|^{N+ps}}\,dy \\
&\ge \frac{(3/2R)^{s(p-1)}}{(R/2)^{N+ps}}\int_{\tilde B_R}\big(\frac{u(y)}{\ds(y)}-m\big)^{p-1}\,dy,
\end{align*}
and thus by H\"older inequality and the fact that $u\ge m\ds$ in $\tilde B_{R}$,
\[\fpl w(x)\le\frac{\lambda^{p-1}}{R^s}-\frac{1}{C R^{s}}\dashint_{\tilde B_R}\Big(\frac{u(y)}{\ds(y)}-m\Big)^{p-1}\,dy\le \frac{\lambda^{p-1}}{R^s}-\frac{\Ex(u, m, R)^{p-1}}{C R^{s}}.\]
Choosing 
\beq
\label{lbdeflambda}
\lambda=\frac{\Ex(u, m, R)}{(2C^2)^\frac{1}{p-1}},
\eeq
 we have
\beq\label{lb3}
\fpl w\le -\frac{\Ex(u, m, R)^{p-1}}{2CR^s} \quad \text{in $D_R$.}
\eeq
Now we choose the constants, setting
\[\theta_1=\frac{1}{\sigma_1}=2C(2C^2)^\frac{1}{p-1}, \quad C_4=\sigma_1\max\,\big\{(4C)^\frac{1}{p-1},\,4C\theta_1^{2-p}\big\}.\]
Clearly $\theta_1,\,C_4\ge 1\ge\sigma_1>0$ only depend on $N$, $p$, $s$. Assuming
\beq\label{lb4}
\Ex(u, m, R)\ge m\theta_1,
\eeq
we claim that
\[\inf_{D_{R/2}}\Big(\frac{u}{\ds}-m\Big) \ge \sigma_1\Ex(u, m, R)-C_4(KR^s)^\frac{1}{p-1}-C_4HR^s.\]
Two cases may occur:
\begin{itemize}[leftmargin=0.6cm]
\item[$(a)$] If $\sigma_1\Ex(u, m, R)\le C_4(KR^s)^\frac{1}{p-1}+C_4HR^s$, then the claim is immediate being the left hand side non-negative.
\item[$(b)$] If $\sigma_1\Ex(u, m, R)>C_4(KR^s)^\frac{1}{p-1}+C_4HR^s$, then from the definitions above and \eqref{lb4} we have
\[\Ex(u, m, R)^{p-1} \ge \begin{cases}
\displaystyle\Big(\frac{C_4}{\sigma_1}\Big)^{p-1}KR^s \ge 4CKR^s \\[2pt]
\displaystyle (m\theta_1)^{p-2}\Ex(u, m, R) \ge (m\theta_1)^{p-2}\frac{C_4}{\sigma_1}HR^s \ge 4Cm^{p-2}HR^s,
\end{cases}\]
and by summing up
\[\Ex(u, m, R)^{p-1}\ge 2CR^s(K+m^{p-2}H).\]
Now by \eqref{super}, \eqref{lb3}, and recalling that $w=\chi_{\tilde B_R}u$ in $D_R^c$, we have
\[\begin{cases}
\fpl w \le K+m^{p-2}H\le \fpl u & \text{in $D_R$} \\
w\le u & \text{in $D_R^c$.}
\end{cases}\]
By Proposition \ref{comp} we have $w\le u$ in $\R^N$. In particular, for all $x\in D_{R/2}$, recalling \eqref{lb2} and the definition of $\lambda$ in \eqref{lbdeflambda}, we have
\[u(x) \ge \frac{\lambda }{R^\frac{s}{p-1}}v(x) \ge \frac{\Ex(u, m , R)}{C(2C^2)^\frac{1}{p-1}}\ds(x).\]
Thus, by \eqref{lb4} again
\[\inf_{D_{R/2}}\Big(\frac{u}{\ds}-m\Big) \ge \Ex(u, m, R)\Big(\frac{1}{C(2C^2)^\frac{1}{p-1}}-\frac{1}{\theta_1}\Big) = \sigma_1\Ex(u, m , R).\]
\end{itemize}
In both cases the proof is concluded.
\end{proof}

\begin{remark}
In Lemma \ref{lowerbig} we bound $u/\ds$ from below by means of the sum of three terms, one of which depends on $u$ while the others do not, and the latter are in fact dropped unless the sum is negative. This strategy will be used several times in the following results.
\end{remark}

\noindent
The next result is a change of variables lemma for $\fpl$, strictly related to the discussion on the boundedness of the fractional $p$-Laplacians of distance functions developed in \cite[Section 3]{IMS1}. Here $\mathcal{GL}_N$ denotes the group of all invertible matrices in $\R^{N\times N}$, and $|A|$ denotes any matrix norm. For all $A\in\mathcal{GL}_N$, $x\in\Omega$, and $\eps>0$ we set
\beq\label{geps}
g_\eps(A, x)=\int_{B_\eps^c(x)}\frac{(\ds(x)-\ds(y))^{p-1}}{|A(x-y)|^{N+ps}}\,dy.
\eeq
We need some more notation for this result: for all $U,V\subset\R^N$ we denote the Hausdorff distance between $U$ and $V$ by
\[{\rm dist}_{\mathcal H}(U,V)=\max\Big\{\sup_{x\in U}{\rm dist}(x,V),\,\sup_{y\in V}{\rm dist}(y,U)\Big\},\]
and the symmetric difference by
\[U\Delta V=(U\setminus V)\cup (V\setminus U).\]
Finally, for all $U\subset\R^N$ we denote by ${\mathcal H}^{N-1}(U)$ the $(N-1)$-dimensional Hausdorff measure of $U$.

\begin{lemma}\label{change}
{\rm (Change of variables)} If $\partial\Omega$ is $C^{1,1}$, there exist $\delta=\delta(N)>0$, $C_5=C_5(N,p,s,\Omega)>0$ s.t.\
\begin{enumroman}
\item $g_\eps\to g_0$ in $L^\infty_{\rm loc}(B_{\delta}(I)\times\Omega_{\rho/2})$, as $\eps\to 0^+$;
\item $\|g_0\|_{L^\infty(B_{\delta}(I)\times\Omega_{\rho/2})}\le C_5$.
\end{enumroman}
\end{lemma}
\begin{proof}
Since $\mathcal{GL}_N$ is an open subset of $\R^{N\times N}$, we can find $\delta>0$ (only depending on $N$) s.t.\ $B_{2\delta}(I)\subset\mathcal{GL}_N$. Choose $A\in B_\delta(I)$, $C=C(N)>0$ s.t.\ $|A|,|A^{-1}|\le C$. By translation invariance and boundedness of $\Omega_{\rho/2}$, we may assume $0\in\partial\Omega$ and prove that $g_\eps\to g_0$ locally uniformly in $B_{\delta}(I)\times D_{\rho/2}$ as $\eps\to 0^+$, for some $g_0$ with $\|g_0\|_{L^\infty(B_{\delta}(I)\times D_{\rho/2})}\le C$ (allowing $C>0$ to grow bigger and eventually depend on $N$, $p$, $s$, and $\Omega$). As the estimates will be uniform with respect to $A\in B_{\delta}(I)$, we will omit the dependence on $A$ for simplicity.
\vskip2pt
\noindent
Observe that restricting the domain of integration in \eqref{geps} to $D_{3\rho/4}\cap B_{\eps}^{c}(x)$ has the sole effect of adding an equi-bounded term to both $g_{\eps}$ and $g_{0}$, so that we can actually prove the statement for 
\[\tilde{g}_{\eps}(x)=\int_{D_{3\rho/4}\cap B_\eps^c(x)}\frac{(\ds(x)-\ds(y))^{p-1}}{|A(x-y)|^{N+ps}}\,dy.\]
\vskip2pt
\noindent
Since $\partial\Omega$ is of class $C^{1,1}$, there exists a diffeomorphism $\Phi\in C^{1,1}(\R^N,\R^N)$ s.t.\ $\Phi(0)=0$, ${\rm d}_\Omega(x)=(x'_N)_+$ for all $x\in D_{\rho/2}$, $x'=\Phi(x)$, and $\Phi(D_{3\rho/4})\subseteq D'_{\rho'}$ where $\rho'=\rho'(\Omega)>0$ and
\[D'_{\rho'}=\big\{x'\in\R^N:\,|x'|<\rho',\,x'_N>0\big\}.\]
Moreover we may assume (taking $C>0$ bigger if necessary) that for all $x\in D_\rho$, $x'\in D'_{\rho'}$
\beq\label{unif}
\frac{1}{C}\le |D\Phi(x)|,\,|D\Phi^{-1}(x')|\le C.
\eeq
Now fix $x\in D_{\rho/2}$ and set 
\[x'=\Phi(x),\quad  M_x=D\Phi^{-1}(x')=(D\Phi(x))^{-1}.\]
Fix as well $\eps\in\ ]0, \rho/4[$. We act on $\tilde{g}_\eps(x)$ with the change of variables $y'=\Phi(y)$ and we get
\begin{align*}
\tilde{g}_\eps(x) &= \int_{\Phi(D_{3\rho/4}\cap B_\eps^c(x))}\frac{((x'_N)_+^s-(y'_N)_+^s)^{p-1}}{|A(\Phi^{-1}(x')-\Phi^{-1}(y'))|^{N+ps}} |{\rm det}D\Phi^{-1}(y')|\,dy' \\
 &= \int_{\Phi(D_{3\rho/4}\cap B_\eps^c(x))}\frac{((x'_N)_+^s-(y'_N)_+^s)^{p-1}}{|AM_x(x'-y')|^{N+ps}}K(x',y')\,dy',
\end{align*}
where we have set
\[K(x',y')=\frac{|AM_x(x'-y')|^{N+ps}|{\rm det}D\Phi^{-1}(y')|}{|A(\Phi^{-1}(x')-\Phi^{-1}(y'))|^{N+ps}}.\]
Again we can add a bounded term to $\tilde{g}_{\eps}$ and instead consider
\beq\label{cv1}
h_{\eps}(x)=\int_{\Phi(B_\eps^c(x))}\frac{((x'_N)_+^s-(y'_N)_+^s)^{p-1}}{|AM_x(x'-y')|^{N+ps}}K(x',y')\,dy'.
\eeq
By \eqref{unif} we have for all $x'\in D'_{\rho'}$, $y'\in\Phi(B_\eps^c(x))$
\[\frac{1}{C}\le K(x',y')\le C.\]
We introduce a linearized operator $L_x:\R^N\to\R^N$ by setting for all $y\in\R^N$
\[L_x(y)=\Phi(x)+D\Phi(x)(y-x),\]
which by Taylor expansion and $\Phi\in C^{1,1}(\R^N)$ satisfies for all $y\in\R^N$
\[|L_x(y)-\Phi(y)|\le C|x-y|^2.\]
In turn, this implies the geometric inequality
\[d_\eps:={\rm dist}_{\mathcal H}\big(\Phi(B_\eps(x)),\,L_x(B_\eps(x))\big)\le C\eps^2.\]
Set
\[\Delta_\eps(x)=\Phi(B_\eps^c(x))\Delta L_x(B_\eps^c(x)),\]
then by the inequality above and $\Phi(B_\eps(x))\Delta L_x(B_\eps(x))\subseteq B_{d_\eps}(\partial L_x(B_\eps(x)))$ we have 
\beq\label{cv2}
|\Delta_\eps(x)|\le C{\mathcal H}^{N-1}(\partial L_x(B_\eps^c(x)))\eps^2 \le C\eps^{N+1}.
\eeq
We split \eqref{cv1} as:
\begin{align*}
h_\eps(x) &= \Big[\int_{\Phi(B_\eps^c(c))\setminus L_x(B_\eps^c(x))}\frac{((x'_N)_+^s-(y'_N)_+^s)^{p-1}}{|AM_x(x'-y')|^{N+ps}}K(x',y')\,dy' \\
&\quad - \int_{L_x(B_\eps^c(x))\setminus\Phi(B_\eps^c(x))}\frac{((x'_N)_+^s-(y'_N)_+^s)^{p-1}}{|AM_x(x'-y')|^{N+ps}}K(x',y')\,dy'\Big] \\
&\quad +\int_{L_x(B_\eps^c(x))}\frac{((x'_N)_+^s-(y'_N)_+^s)^{p-1}}{|AM_x(x'-y')|^{N+ps}}K(x',y')\,dy' \\
&= h_\eps^1(x)+h_\eps^2(x),
\end{align*}
and we estimate separately $h_\eps^1(x)$ and $h_\eps^2(x)$. By $s$-H\"older continuity of the function $y'\to (y'_N)_+^s$, estimates on $|M_x|$, \eqref{cv2}, and direct integration we have
\begin{align*}
|h_\eps^1(x)| &\le C\int_{\Delta_\eps(x)}\frac{|(x'_N)_+^s-(y'_N)_+^s|^{p-1}}{|AM_x(x'-y')|^{N+ps}}\,dy' \\
&\le C|A^{-1}|^{N+ps}\int_{\Delta_\eps(x)}\frac{dy'}{|x'-y'|^{N+s}} \le \frac{C|{\rm d}_{\eps}(x)|}{\eps^{N+s}}\le C\eps^{1-s},
\end{align*}
which by $s\in\ ]0,1[$ implies
\beq\label{cv3}
h_\eps^1\to 0 \quad \text{in $L^\infty(D_{\rho/2})$ as $\eps\to 0^{+}$.}
\eeq
Now we turn to $h_\eps^2(x)$, which we split further:
\begin{align*}
h_\eps^2(x) &= \int_{L_x(B_\eps^c(x))}\frac{((x'_N)_+^s-(y'_N)_+^s)^{p-1}}{|AM_x(x'-y')|^{N+ps}}|{\rm det}M_x|\,dy' \\
&\quad+ \int_{L_x(B_\eps^c(x))}\frac{((x'_N)_+^s-(y'_N)_+^s)^{p-1}}{|AM_x(x'-y')|^{N+ps}}(K(x',y')-|{\rm det}M_x|)\,dy' \\
&= h_\eps^3(x)+h_\eps^4(x).
\end{align*}
We first deal with $h_\eps^3(x)$, using the change of variables $z'=M_x(x'-y')$ and setting $x''=M_x x'$, $\lambda=|M_x^{-T}e_N|$, and $v'=\lambda^{-1}M_x^{-T}e_N$:
\begin{align*}
h_\eps^3(x) &= \int_{B_\eps^c}\frac{\big(((M_x^{-1}x'')\cdot e_N)_+^s-(M_x^{-1}(x''-z')\cdot e_N)_+^s\big)^{p-1}}{|Az'|^{N+ps}}\,dz' \\
&= \lambda^{(p-1)s}\int_{B_\eps^c}\frac{\big((x'\cdot v')_+^s-((x'-z')\cdot v')_+^s\big)^{p-1}}{|Az'|^{N+ps}}\,dz'.
\end{align*}
By rotational invariance and \cite[Lemma 3.2]{IMS1}, we have
\beq\label{cv4}
h_\eps^3\to 0 \quad \text{in $L^\infty_{\rm loc}(D_{\rho/2})$ as $\eps\to 0^{+}$}
\eeq
(this is where the convergence turns locally uniform instead of uniform).
To estimate $h_\eps^4(x)$, we can again add a bounded term and consider instead
\[h_{\eps}^{5}(x):=\int_{L_x(B_1(x))\setminus L_x(B_\eps(x))}\frac{((x'_N)_+^s-(y'_N)_+^s)^{p-1}}{|AM_x(x'-y')|^{N+ps}}(K(x',y')-|{\rm det}M_x|)\,dy'.\]
By \cite[Eq.\ (3.7)]{IMS1} we have
\[\big|K(x',y')-|{\rm det}M_x|\big|\le C|x'-y'| \quad \text{for all $x'\in D'_{\rho'}$, $y'\in L_x(B_\eps^c(x))$}\]
and using H\"older continuity, we have
\[\frac{|(x'_N)_+^s-(y'_N)_+^s|^{p-1}}{|AM_x(x'-y')|^{N+ps}}|K(x',y')-|{\rm det}M_x||\le C\frac{|x'-y'|^{(p-1)s+1}}{|AM_x(x'-y')|^{N+ps}} \le \frac{C}{|x'-y'|^{N+s-1}},\]
and the latter function lies in $L^1(L_x(B_1(x)))$. Now, letting
\[h^{5}(x):=\int_{L_x(B_1(x))\setminus L_x(B_\eps(x))}\frac{((x'_N)_+^s-(y'_N)_+^s)^{p-1}}{|AM_x(x'-y')|^{N+ps}}(K(x',y')-|{\rm det}M_x|)\,dy'.\]
we have, via direct integration and $L_{x}(B_{1}(x))\subseteq B_{C}(x')$,
\[|h^{5}(x)|\le C\int_{L_x(B_1(x))\setminus L_x(B_\eps(x))}\frac{dy'}{|x'-y'|^{N+s-1}},\]
and similarly, by $L_{x}(B_{\eps}(x))\subseteq B_{C\eps}(x')$ (see \eqref{unif}),
\[|h^{5}_{\eps}(x)-h^{5}(x)|\le C\int_{L_x(B_\eps(x))}\frac{dy'}{|x'-y'|^{N+s-1}}\le C\eps^{1-s}.\]
Again by $s\in(0,1)$, we deduce that $h^5_\eps\to h^5$ in $L^\infty_{\rm loc}(D_{\rho/2})$ as $\eps\to 0^+$. Taking into account the several splittings and \eqref{cv3}, \eqref{cv4}, we obtain the claim.
\end{proof}

\noindent
By virtue of the previous result, we are able to construct our first  barrier:

\begin{lemma}\label{lowerbarrier}
{\rm (Barrier/1)} Let $\partial\Omega$ be $C^{1,1}$, $R\in\ ]0,\rho/4[$, $\varphi\in C^\infty_c(B_1)$ s.t.\ $0\le\varphi\le 1$ in $B_1$, and for all $\lambda\in\R$, $x\in\R^N$ set
\[w_\lambda(x)=\Big(1+\lambda\varphi\Big(\frac{2x}{R}\Big)\Big)\ds(x).\]
Then, there exist $\lambda_1=\lambda_1(N,p,s,\Omega,\varphi)>0$, $C_6=C_6(N,p,s,\Omega,\varphi)>0$ s.t.\ for all $|\lambda|\le\lambda_1$
\[|\fpl w_\lambda|\le C_6\Big(1+\frac{|\lambda|}{R^s}\Big) \quad \text{in $D_{R/2}$.}\]
\end{lemma}
\begin{proof}
For $\lambda=0$, the conclusion follows from \cite[Theorem 3.6]{IMS1}. So, let $\lambda\in\R$ satisfy
\[0<|\lambda|\le\frac{1}{2\|\varphi\|_{L^\infty(B_1)}}.\]
Set for all $x\in\R^N$
\[\psi_\lambda(x)=\frac{1}{\lambda}\Big(\Big(1+\lambda\varphi\Big(\frac{2x}{R}\Big)\Big)^\frac{1}{s}-1\Big),\]
so $\psi_\lambda\in C^\infty_c(B_{R/2})$ and for all $x\in\R^N$
\[1+\lambda\varphi\Big(\frac{2x}{R}\Big)=(1+\lambda\psi_\lambda(x))^s.\]
Moreover, by the chain rule there exists $C>0$ (depending on $N$, $p$, $s$, $\Omega$, and $\varphi$) s.t.\ for all $x\in\R^N$
\beq\label{lbar1}
|\psi_\lambda(x)|+R|D\psi_\lambda(x)|+R^2|D^2\psi_\lambda(x)|\le C\chi_{B_{R/2}}(x).
\eeq
Since $\Pi_\Omega\in C^{1,1}(B_R, \partial\Omega)$, by taking $|\lambda|>0$ even smaller if necessary, we may set for all $x\in\R^N$
\[\Phi_\lambda(x)=x+\lambda\psi_\lambda(x)(x-\Pi_\Omega(x)),\]
thus defining a diffeomorphism $\Phi_\lambda\in C^{1,1}(\R^N,\R^N)$ s.t.\ $\Phi_\lambda(\Omega)=\Omega$, $\Phi_\lambda(x)=x$ for all $x\in B_{R/2}^{c}$, and $\Pi_\Omega(\Phi_\lambda(x))=\Pi_\Omega(x)$ for all $x\in D_R$. Besides we define $\Psi_\lambda=\Phi_\lambda^{-1}\in C^{1,1}(\R^N,\R^N)$. The key point is that $w_{\lambda}$ is actually equivalent to a distance function, up to the diffeomorphism $\Phi_{\lambda}$ of the domain. Indeed, with these notations, we have for all $x\in D_R$
\beq\label{lbar4}
\begin{split}
w_\lambda(x) &= (1+\lambda\psi_\lambda(x))^s|x-\Pi_\Omega(x)|^s = |\Phi_\lambda(x)-\Pi_\Omega(x)|^s \\
&= |\Phi_\lambda(x)-\Pi_\Omega(\Phi_\lambda(x))|^s = \ds(\Phi_\lambda(x)),
\end{split}
\eeq
We begin collecting some estimates on the first and second order derivatives of $\Phi_\lambda$, $\Psi_\lambda$ that will be used later.  For all $x, x'\in\R^N$ we claim
\beq\label{lbar2}
|D\Phi_\lambda(x)-I|\le C|\lambda|\chi_{B_{R/2}}(x),\quad |D\Psi_\lambda(x')-I|\le C|\lambda|\chi_{B_{R/2}}(x').
\eeq
Indeed, recall that $\Phi_\lambda=\Psi_\lambda=I$ in $B_{R/2}^c$. Instead, for all $x\in B_{R/2}$, $i,k\in\{1,\ldots N\}$ we have
\[\partial_i\Phi_\lambda^k(x)-\delta_{ik} = \lambda \Big[\partial_i\psi_\lambda(x)(x-\Pi_\Omega(x))^k+\psi_\lambda(x)(\delta_{ik}-\partial_i\Pi_\Omega^k(x))\Big]\]
(where $\xi^k$ denotes the $k$-th component of $\xi\in\R^N$, $\delta_{ik}$ is the Kronecker symbol and $\partial_{i}$ is the partial derivative with respect to $x_{i}$). By \eqref{lbar1}, this implies the first inequality in \eqref{lbar2}.
By further reducing $|\lambda|>0$ if necessary, the latter yields $|(D\Phi_\lambda(x))^{-1}|\le C$, hence, setting $x'=\Phi_\lambda(x)$,
\[|D\Psi_\lambda(x')-I| = |(D\Phi_\lambda(x))^{-1}(I-D\Phi_\lambda(x))| \le C|\lambda|\chi_{B_{R/2}}(x)=C|\lambda|\chi_{B_{R/2}}(x'),\]
(since $\Phi_{\lambda}(B_{R/2})=B_{R/2}$), which concludes the proof \eqref{lbar2}. 

\noindent
Regarding the second-order derivatives, for a. e. $x,x'\in\R^N$ we claim
\beq\label{lbar3}
|D^2\Phi_\lambda(x)|\le\frac{C|\lambda|}{R}\chi_{B_{R/2}}(x), \quad |D^2\Psi_\lambda(x')|\le\frac{C|\lambda|}{R}\chi_{B_{R/2}}(x').
\eeq
Indeed, for a. e.\ $x\in B_{R/2}$, $i,j,k\in\{1,\ldots N\}$ we have
\begin{align*}
\partial_{ij}\Phi^k_\lambda(x) &= \lambda \Big[ \partial_{ij}\psi_\lambda(x)(x-\Pi_\Omega(x))^k+ \partial_i\psi_\lambda(x)(\delta_{jk}-\partial_j\Pi_\Omega^k(x)) \\
&\qquad + \partial_j\psi_\lambda(x)(\delta_{ik}-\partial_i\Pi_\Omega^k(x))+\psi_\lambda(x)\partial_{ij}\Pi_\Omega^k(x)\Big],
\end{align*}
which by \eqref{lbar1} implies the first estimate in \eqref{lbar3}.
Regarding the second one, observe that $D^{2}\Psi_{\lambda}\equiv 0$ in $B_{R/2}^{c}$, while for $\Phi_{\lambda}(x)=x'\in B_{R/2}$, the chain rule gives, almost everywhere, 
\[
\partial_{ij}\Psi^{k}_{\lambda}(x')=-\partial_{\beta\gamma}\Phi^{\alpha}_{\lambda}(x)\partial_{i}\Psi^{\beta}_{\lambda}(x')\partial_{j}\Psi^{\gamma}_{\lambda}(x')\partial_{\alpha}\Psi^{k}_{\lambda}(x'),
\]
with the sum over repeated indexes convention. Due to the estimate for $D^{2}\Phi_{\lambda}$ and (from \eqref{lbar2}) $\|D\Psi_{\lambda}\|_{\infty}\le C$ when $|\lambda|$ is sufficiently small, we infer the second inequality in \eqref{lbar3}.

\noindent
Now set for all $\eps>0$, $x\in\R^N$
\[g_{\eps, \lambda}(x)=\int_{\{|\Phi_\lambda(x)-\Phi_\lambda(y)|\ge\eps\}}\frac{(\ds(\Phi_\lambda(x))-\ds(\Phi_\lambda(y)))^{p-1}}{|x-y|^{N+ps}}\,dy.\]
We claim that there exist $\lambda_{1},C>0$, depending only $N$, $p$, $s$, $\Omega$, and $\varphi$, s.t.\ for every $0<|\lambda|<\lambda_{1}$ there exists $g_{0, \lambda}\in L^\infty(D_{R/2})$ s.t.\
\beq\label{lbar5}
\|g_{0, \lambda}\|_{L^\infty(D_{R/2})}\le C\Big(1+\frac{|\lambda|}{R^s}\Big),\quad g_{\eps, \lambda}\to g_{0, \lambda} \ \text{in $L^\infty_{\rm loc}(D_{R/2})$, as $\eps\to 0^+$}.
\eeq
The path to \eqref{lbar5} begins with the change of variables $x'=\Phi_\lambda(x)$, $y'=\Phi_\lambda(y)$ (note that by the previous discussion $x'\in D_{R/2}$ whenever $x\in D_{R/2}$) and defining 
\[K(x',y'):=\frac{|D\Psi_\lambda(x')(x'-y')|^{N+ps}}{|\Psi_\lambda(x')-\Psi_\lambda(y')|^{N+ps}} |{\rm det}D\Psi_\lambda(y')|,\]
so that for all $x\in D_{R/2}$ we can rephrase
\begin{align*}
g_{\eps, \lambda}(x) &= \int_{B_\eps^c(x')}\frac{(\ds(x')-\ds(y'))^{p-1}}{|D\Psi_\lambda(x')(x'-y')|^{N+ps}}K(x',y')\,dy' \\
&= \int_{B_\eps^c(x')}\frac{(\ds(x')-\ds(y'))^{p-1}}{|D\Psi_\lambda(x')(x'-y')|^{N+ps}}|{\rm det}D\Psi_\lambda(x')|\,dy' \\
&\quad+ \int_{B_\eps^c(x')}\frac{(\ds(x')-\ds(y'))^{p-1}}{|D\Psi_\lambda(x')(x'-y')|^{N+ps}}(K(x',y')-|{\rm det}D\Psi_\lambda(x')|)\,dy' \\
&= g_{\eps,\lambda}^1(x)+g_{\eps,\lambda}^2(x).
\end{align*}
By \eqref{lbar2} and Lemma \ref{change}, taking if necessary $|\lambda|>0$ even smaller, the claim \eqref{lbar5} is true for $g_{\eps,\lambda}^1$, with corresponding $g_{0,\lambda}^{1}$ obeying $\|g_{0,\lambda}^{1}\|_{L^{\infty}(D_{R/2})}\le C$. Regarding $g_{\eps,\lambda}^2$, we split as follows:
\begin{align*}
K(x',y')-|{\rm det}D\Psi_\lambda(x')| &= \frac{|D\Psi_\lambda(x')(x'-y')|^{N+ps}}{|\Psi_\lambda(x')-\Psi_\lambda(y')|^{N+ps}}\big(|{\rm det}D\Psi_\lambda(y')|-|{\rm det}D\Psi_\lambda(x')|\big) \\
&\quad+ \Big(\frac{|D\Psi_\lambda(x')(x'-y')|^{N+ps}}{|\Psi_\lambda(x')-\Psi_\lambda(y')|^{N+ps}}-1\Big)|{\rm det}D\Psi_\lambda(x')| \\
&= K_1(x',y')+K_2(x',y').
\end{align*}
We first estimate $K_1(x',y')$, by applying the triangle inequality, Jacobi's formula for the derivative of a determinant, and estimates \eqref{lbar2}, \eqref{lbar3}:
\begin{align*}
|K_1(x',y')| &\le \|D\Phi_\lambda\|_{L^\infty(\R^N)}^{N+ps}\|D\Psi_\lambda\|_{L^\infty(\R^N)}^{N+ps}\big|{\rm det}D\Psi_\lambda(y')-{\rm det}D\Psi_\lambda(x')\big| \\
&\le C\Big|\int_0^1\frac{d}{dt}{\rm det}D\Psi_\lambda(x'+t(y'-x'))\,dt\Big| \\
&\le C\int_0^1\|D\Psi_\lambda\|_{L^\infty(\R^N)}^{N-1}|D^2\Psi_\lambda(x'+t(y'-x'))||x'-y'|\,dt \\
&\le \frac{C|\lambda|}{R}|x'-y'|\int_0^1\chi_{B_{R/2}}(x'+t(y'-x'))\,dt \\
&\le \frac{C|\lambda|}{R}|x'-y'|\min\Big\{1,\frac{R}{|x'-y'|}\Big\} \le C|\lambda|\min\Big\{\frac{|x'-y'|}{R},1\Big\},
\end{align*}
where the calculations above are justified for a.e.\ $y'\in \R^{N}$ since, by a well known property, Sobolev functions (${\rm det}D\Psi_\lambda$ in our case) are absolutely continuous on almost every line. 
Similarly, to estimate $K_2(x',y')$ we argue as in \cite[Lemma 3.4]{IMS1}, applying \eqref{lbar2}, \eqref{lbar3}, and Taylor's expansion with integral remainder:
\begin{align*}
|K_2(x',y')| &\le C\Big(\frac{|D\Psi_\lambda(x')(x'-y')|^2}{|\Psi_\lambda(x')-\Psi_\lambda(y')|^2}-1\Big) \\
&\le C\frac{\big|\Psi_\lambda(x')-\Psi_\lambda(y')+D\Psi_\lambda(x')(x'-y')\big|\,\big|\Psi_\lambda(x')-\Psi_\lambda(y')-D\Psi_\lambda(x')(x'-y')\big|}{|\Psi_\lambda(x')-\Psi_\lambda(y')|^2} \\
&\le \frac{C\|D\Phi_\lambda\|_{L^\infty(\R^N)}^2}{|x'-y'|^2}\big(2\|D\Psi_\lambda\|_{L^\infty(\R^N)}|x'-y'|\big)\Big|\int_0^1(1-t)\frac{d^2}{dt^2}\Psi_\lambda(x'+t(y'-x'))\,dt\Big| \\
&\le \frac{C}{|x'-y'|}\int_0^1\frac{|\lambda||x'-y'|^2}{R}\chi_{B_{R/2}}(x'+t(y'-x'))\,dt \\
&\le \frac{C|\lambda||x'-y'|}{R}\min\Big\{1,\frac{R}{|x'-y'|}\Big\} \le C|\lambda|\min\Big\{\frac{|x'-y'|}{R},1\Big\}.
\end{align*}
Summing up the last relations, we have for all $x'\in D_{R/2}$ and a.e.\ $y'\in\R^N$
\[\big|K(x',y')-|{\rm det}D\Psi_\lambda(x')|\big|\le C|\lambda|\min\Big\{\frac{|x'-y'|}{R},1\Big\}.\]
Set
\[h(x',y')=\frac{|\ds(x')-\ds(y')|^{p-1}}{|D\Psi_\lambda(x')(x'-y')|^{N+ps}}\big|K(x',y')-|{\rm det}D\Psi_\lambda(x')|\big|,\]
so that
\[g_{\eps,\lambda}^2(x)=\int_{B_\eps^c(x')}h(x',y')\,dy'.\]
Using the previous estimate and the $s$-H\"older continuity of $\ds$, we get for all $x'\in D_{R/2}$
\begin{align*}
\|h(x',\cdot)\|_{L^1(\R^N)} 
&\le C|\lambda|\int_{\R^N}\frac{|\ds(x')-\ds(y')|^{p-1}}{|D\Psi_\lambda(x')(x'-y')|^{N+ps}}\min\Big\{\frac{|x'-y'|}{R},1\Big\}\,dy' \\
&\le C|\lambda|\|D\Phi_\lambda\|_{L^\infty(\R^N)}^{N+ps}\int_{\R^N}\frac{1}{|x'-y'|^{N+s}}\min\Big\{\frac{|x'-y'|}{R},1\Big\}\,dy' \\
&\le C|\lambda|\Big[\int_{\{|z'|<R\}}\frac{dz'}{|z'|^{N+s-1}}+\int_{\{|z'|\ge R\}}\frac{dz'}{|z'|^{N+s}}\Big] \le \frac{C|\lambda|}{R^s}.
\end{align*}
By an entirely similar argument to the one used to deal with $h^{5}$ in the previous Lemma, we obtain the claim \eqref{lbar5} for $g^{2}_{\eps,\lambda}$ as well, with corresponding $g^{2}_{0,\lambda}$ obeying $\|g^{2}_{0,\lambda}\|_{L^{\infty}(D_{R/2})}\le C|\lambda|/R^{s}$. Finally, recalling \eqref{lbar4} and applying \cite[Corollary 2.7]{IMS1}, we conclude that, whenever $|\lambda|\le\lambda_1$,
\[\fpl w_{\lambda}(x)=\lim_{\eps\to 0^{+}} g_{\eps}(x),\]
and therefore 
\[|\fpl w_\lambda|\le \|g_{0}\|_{L^{\infty}(D_{R/2})}\le C_6\Big(1+\frac{|\lambda|}{R^s}\Big) \quad \text{in $D_{R/2}$,}\]
for convenient $\lambda_1,C_6>0$ depending on $N$, $p$, $s$, $\Omega$, and $\varphi$.
\end{proof}

\begin{remark}
In the case when $\Omega$ is a half space, we get the cleaner estimate
\[|\fpl w_{\lambda}|\le \frac{C}{R^{s}}|\lambda| \quad \text{in $D_{R/2}$,}\]
for all sufficiently small $|\lambda|$ depending on $\varphi$.
\end{remark}

\noindent
The next result yields a lower bound on the supersolution of \eqref{super} similar to that given in Lemma \ref{lowerbig}, but for small excess (defined in \eqref{lq}):

\begin{lemma}\label{lowersmall}
Let $\partial\Omega$ be $C^{1,1}$, $u\in\widetilde{W}^{s,p}(D_R)$ solve \eqref{super} and $p\ge 2$. Then, for all $\theta\ge 1$ there exist $C_\theta=C_\theta(N,p,s,\Omega,\theta)>1$, $\sigma_\theta=\sigma_\theta(N,p,s,\Omega,\theta)\in(0,1]$ s.t.\ for all $R\in \ ]0, \rho/4[$
\[\Ex(u, m , R)\le m\theta \quad \Longrightarrow \quad \inf_{D_{R/2}}\Big(\frac{u}{\ds}-m\Big)\ge\sigma_\theta\Ex(u, m , R)-C_\theta(m^{p-1}+K)^\frac{1}{p-1}R^\frac{s}{p-1}-C_\theta HR^s.\]
\end{lemma}
\begin{proof}
Let $\varphi\in C^\infty_c(B_1)$ be s.t.\ $0\le\varphi\le 1$, $\varphi= 1$ in $ B_{1/2}$, and set for all $\lambda>0$, $x\in\R^N$
\[w_\lambda(x)=m\Big(1+\lambda\varphi\Big(\frac{x}{R}\Big)\Big)\ds(x).\]
Then $w_\lambda\in W^{s,p}(D_R)$ and satisfies
\beq\label{ls1}
\inf_{\tilde B_R}w_\lambda\ge m\Big(\frac{3R}{2}\Big)^s, \quad \sup_{D_R}w_\lambda\le m(1+\lambda)R^s,
\eeq
where $\tilde B_R$ is defined as in \eqref{bt}. By homogeneity and Lemma \ref{lowerbarrier} (with $R$ in the place of $R/2$) we can find $\lambda_1>0$ and $C_6>1$ s.t.\ for all $\lambda\in(0,\lambda_1]$
\beq\label{ls2}
\fpl w_\lambda\le C_6 m^{p-1}\Big(1+\frac{\lambda}{R^s}\Big) \quad \text{in $D_R$.}
\eeq
With no loss of generality we may assume
\[0<\lambda_1\le\min\Big\{1,\frac{(3/2)^s-1}{2}\Big\}.\]
Now set for all $x\in\R^N$
\[v_\lambda(x)=\begin{cases}
w_\lambda(x) & \text{if $x\in\tilde B_R^c$} \\
u(x) & \text{if $x\in\tilde B_R$.}
\end{cases}\]
Clearly, since $\tilde B_R$ is bounded and at a positive distance from $D_R$, we can apply Proposition \ref{spp} and deduce that $v_\lambda\in\widetilde{W}^{s,p}(D_R)$ and for all $x\in D_R$
\beq\label{ls3}
\fpl v_\lambda(x) = \fpl w_\lambda(x)+2\int_{\tilde B_R}\frac{(w_\lambda(x)-u(y))^{p-1}-(w_\lambda(x)-w_\lambda(y))^{p-1}}{|x-y|^{N+ps}}\,dy.
\eeq
We need to estimate the integral in \eqref{ls3}. We note that, for all $x\in D_R$ and $y\in\tilde B_R$, by \eqref{super} we have $u(y)\ge m\ds(y)\ge w_\lambda(y)$. Using \eqref{ls1}, we have as well
\[u(y)-w_\lambda(x) \ge w_\lambda(y)-w_\lambda(x) \ge \frac{mR^s}{2}\Big(\Big(\frac{3}{2}\Big)^s-1\Big).\]
By Lagrange's theorem we deduce
\[(u(y)-w_\lambda(x))^{p-1}-(w_\lambda(y)-w_\lambda(x))^{p-1} \ge \frac{m^{p-2}R^{(p-2)s}}{C}(u(y)-m\ds(y)).\]
Plugging \eqref{ls2} and these estimates into \eqref{ls3} and recalling the properties \eqref{bt} of $\tilde B_{R}$, we get 
\begin{align*}
\fpl v_\lambda(x) &\le C_6m^{p-1}\Big(1+\frac{\lambda}{R^s}\Big)-2\frac{m^{p-2}R^{(p-2)s}}{C}\int_{\tilde B_R}\frac{u(y)-m\ds(y)}{|x-y|^{N+ps}}\,dy \\
&\le C_6m^{p-1}\Big(1+\frac{\lambda}{R^s}\Big)-\frac{2m^{p-2}}{CR^s}\dashint_{\tilde B_R}\Big(\frac{u(y)}{\ds(y)}-m\Big)\,dy \\
&\le Cm^{p-1}+\frac{m^{p-2}}{R^s}\Big(C\lambda m-\frac{\Ex(u, m , R)}{C}\Big),
\end{align*}
for all $x\in D_R$. We then want to find suitable $\sigma_{\theta}, C_{\theta}, \lambda$ s.t.\ either the thesis is trivial, or
\[Cm^{p-1}+\frac{m^{p-2}}{R^s}\Big(C\lambda m-\frac{\Ex(u, m , R)}{C}\Big)\le -K-m^{p-2}H,\]
allowing by comparsion to infer $u\ge w_{\lambda}$ in $D_{R}$. As it turns out, this reduces to an elementary set of inequalities, which can be solved for $\lambda$ being the right quantity to get the conclusion.
\vskip2pt
\noindent
We thus fix $\theta\ge 1$, set
\[\sigma_\theta=\frac{\lambda_1}{2\theta C^2},  \quad C_\theta=\sigma_\theta\max\big\{4C,\,(4C^2\theta^{p-2})^\frac{1}{p-1}\big\}, \quad \lambda=\frac{\sigma_\theta \Ex(u, m , R)}{m},\]
and assume
\beq\label{ls4}
\Ex(u, m , R)\le m\theta.
\eeq
By the choice of constants and \eqref{ls4} we have
\[\lambda\le\frac{\lambda_1}{2C^2}<\lambda_1, \quad C\lambda m\le\frac{\Ex(u, m , R)}{2C},\]
so by the estimate above
\beq\label{ls5}
\fpl v_\lambda\le Cm^{p-1}-\frac{m^{p-2}\Ex(u, m , R)}{2CR^s} \quad \text{in $D_R$.}
\eeq
Being the left-hand side of the thesis non-negative by assumption, we can suppose
\[\sigma_\theta\Ex(u, m , R)\ge C_\theta(m^{p-1}+K)^\frac{1}{p-1}R^\frac{s}{p-1}+C_\theta HR^s.\]
In particular, by the choice of $C_{\theta}$ and \eqref{ls4} (recall that $C\ge 1\ge \sigma_{\theta}$)
\begin{align*}
\Ex(u, m , R)^{p-1} &\ge \frac{C_\theta^{p-1}(m^{p-1}+K)R^s}{\sigma_\theta^{p-1}} \ge 4C^{2}\theta^{p-2}(m^{p-1}+K)R^s \\
&\ge \frac{4C\Ex(u, m , R)^{p-2}}{m^{p-2}}(Cm^{p-1}+K)R^s.
\end{align*}
The last two inequalities lead to
\[m^{p-2}\Ex(u, m , R)\ge\begin{cases}
4C(Cm^{p-1}+K)R^s \\[3pt]
\displaystyle \frac{C_\theta}{\sigma_\theta}m^{p-2}HR^s \ge 4Cm^{p-2}HR^s,
\end{cases}\]
and by summing up to
\[m^{p-2}\Ex(u, m , R)\ge 2C(Cm^{p-1}+K+m^{p-2}H)R^s.\]
Thus, by \eqref{ls5} and \eqref{super} we have
\[\begin{cases}
\fpl v_\lambda\le -K-m^{p-2}H \le\fpl u & \text{in $D_R$} \\
v_\lambda\le u & \text{in $D_R^c$,}
\end{cases}\]
which by Proposition \ref{comp} implies $v_\lambda\le u$ in $\R^N$. In particular, recalling the definitions of $w_\lambda$ and $\lambda$, for all $x\in D_{R/2}$ we have
\[\frac{u(x)}{\ds(x)}-m \ge \frac{w_\lambda(x)}{\ds(x)}-m = m\lambda \ge \sigma_\theta\Ex(u, m , R),\]
which gives the conclusion.
\end{proof}

\noindent
Finally, we localize  the global bound from below in \eqref{super} and prove the main result of this section, i.e., the lower bound on supersolutions of \eqref{dp} type problems {\em locally} bounded from below by a multiple of $\ds$. Precisely, we deal, for some $\tilde K,m\ge 0$, with the problem
\beq\label{superloc}
\begin{cases}
\fpl u\ge -\tilde K & \text{in $D_R$} \\
u\ge m\ds & \text{in $D_{2R}$.}
\end{cases}
\eeq

\begin{proposition}\label{lower}
{\rm (Lower bound)} Let $\partial\Omega$ be $C^{1,1}$, $u\in\widetilde{W}^{s,p}_0(D_R)$ solve \eqref{superloc} and $p\ge 2$. Then, for all $\eps>0$ there exist $\tilde C_\eps=\tilde C_\eps(N,p,s,\Omega,\eps)>0$ and two more constants $\sigma_2=\sigma_2(N,p,s,\Omega)\in(0,1]$, $C_7=C_7(N,p,s,\Omega)>1$ s.t.\ for all $R\in \ ]0, \rho/4[$
\begin{align*}
\inf_{D_{R/2}}\Big(\frac{u}{\ds}-m\Big) &\ge \sigma_2\Ex(u, m , R)-\eps\Big\|\frac{u}{\ds}-m\Big\|_{L^\infty(D_R)}  - C_7{\rm tail}_1\Big(\Big(m-\frac{u}{\ds}\Big)_+,2R\Big)R^s\\
&\quad-\tilde C_\eps\Big[m+\tilde K^\frac{1}{p-1}+{\rm tail}_{p-1}\Big(\Big(m-\frac{u}{\ds}\Big)_+,2R\Big)\Big]R^\frac{s}{p-1}.
\end{align*}
\end{proposition}
\begin{proof}
We may assume $u/\ds-m\in L^\infty(D_R)$, otherwise there is nothing to prove. We set $v=u\vee m\ds$ and fix $\eps>0$. By \eqref{superloc} and Proposition \ref{updown} \ref{updown1} (with $\eps^{p-1}$ replacing $\eps$) there exists $C_\eps,C_3>0$ with $C_3$ depending on $N$, $p$, $s$, and $C_\eps$ also depending on $\eps$, s.t.\ in $D_R$
\begin{align*}
\fpl v &\ge -\tilde K-\frac{\eps^{p-1}}{R^s}\Big\|\frac{u}{\ds}-m\Big\|_{L^\infty(D_R)}^{p-1}-C_\eps{\rm tail}_{p-1}\Big(\Big(m-\frac{u}{\ds}\Big)_+,2R\Big)^{p-1} \\
&\quad -C_3m^{p-2}{\rm tail}_1\Big(\Big(m-\frac{u}{\ds}\Big)_+,2R\Big) =: -K-m^{p-2}H,
\end{align*}
where we have set
\[K=\tilde K+\frac{\eps^{p-1}}{R^s}\Big\|\frac{u}{\ds}-m\Big\|_{L^\infty(D_R)}^{p-1}+C_\eps{\rm tail}_{p-1}\Big(\Big(m-\frac{u}{\ds}\Big)_+,2R\Big)^{p-1},\]
\[H=C_3{\rm tail}_1\Big(\Big(m-\frac{u}{\ds}\Big)_+,2R\Big).\]
Thus, $v$ satisfies \eqref{super} with $K,H,m\ge 0$ defined as above. By Lemma \ref{lowerbig} we can find constants $0<\sigma_1\le 1\le\theta_1,\,C_4$ (depending on $N$, $p$, $s$) s.t.\
\[\Ex(v, m , R)\ge m\theta_1 \ \Longrightarrow \ \inf_{D_{R/2}}\Big(\frac{v}{\ds}-m\Big) \ge \sigma_1\Ex(v, m , R)-C_4(KR^s)^\frac{1}{p-1}-C_4HR^s.\]
Next, choose $\theta=\theta_1\ge 1$ in Lemma \ref{lowersmall}. Then, there exist constants $0<\sigma_{\theta_1}\le 1\le C_{\theta_1}$  s.t.\
\[\Ex(v, m , R)\le m\theta_1 \ \Longrightarrow \ \inf_{D_{R/2}}\Big(\frac{v}{\ds}-m\Big) \ge \sigma_{\theta_1}\Ex(v, m , R)-C_{\theta_1}(m^{p-1}+K)^\frac{1}{p-1}R^\frac{s}{p-1}-C_{\theta_1}HR^s.\]
Set $\sigma_{2}\in \ ]0, 1[$, $C>1$ defined as
\[\sigma_2=\min\big\{\sigma_1,\,\sigma_{\theta_1}\big\}, \quad C=\max\big\{C_4,\,C_{\theta_1}\big\},\]
hence depending only on $N$, $p$, $s$, and $\Omega$. 
In both cases, since $v=u$ in $D_{2R}\supset \tilde B_{R}$, we have 
\[\inf_{D_{R/2}}\Big(\frac{u}{\ds}-m\Big) \ge \sigma_{2}\Ex(u, m , R)-C(m^{p-1}+K)^\frac{1}{p-1}R^\frac{s}{p-1}-CHR^s.\]
By \eqref{superloc} and the definitions of $K$, $H$, we have
\begin{align*}
\inf_{D_{R/2}}\Big(\frac{u}{\ds}-m\Big) &\ge \sigma_2\Ex(u, m , R)- CK^\frac{1}{p-1}R^\frac{s}{p-1}- CHR^s\\
&\ge \sigma_2\Ex(u, m , R)-C\eps\Big\|\frac{u}{\ds}-m\Big\|_{L^\infty(D_R)} -  C{\rm tail}_1\Big(\Big(m-\frac{u}{\ds}\Big)_+,2R\Big)R^s\\
&\quad -C\Big[m+\tilde K^\frac{1}{p-1}+C_\eps{\rm tail}_{p-1}\Big(\Big(m-\frac{u}{\ds}\Big)_+,2R\Big)\Big]R^\frac{s}{p-1}, 
\end{align*}
which gives the claim (by renaming $\eps$ and the constants involved)
\end{proof}

\section{The upper bound}\label{sec4}

\noindent
This section is devoted to proving an upper bound for the quotient $u/\ds$, where $u$ is a subsolution of a \eqref{dp} type problem, locally bounded from above by a multiple of $\ds$. The upper bound differs substantially from the lower one, as for large values of the corresponding nonlocal excess, the function $u$ will change sign along the boundary, which of course agrees with $u$ being bounded {\em from above} by a positive multiple of $\ds$. The difficulty comes then from the degeneracy of $\fpl$, as $u$ will have vanishing normal $s$-derivative at some boundary point, and any barrier for $u$ forcing such transition will present the same phenomenon and thus require a more delicate construction. 
\vskip2pt
\noindent
Throughout, we will assume $0\in\partial\Omega$, $R\in\ ]0,\rho/4[$ with $\rho$ defined in \eqref{grho}. As in Section \ref{sec3}, we first consider a function $u\in\widetilde{W}^{s,p}(D_R)$ satisfying
\beq\label{sub}
\begin{cases}
\fpl u\le K+M^{p-2}H & \text{in $D_R$} \\
u\le M\ds & \text{in $\R^N$,}
\end{cases}
\eeq
for some $M,K,H\ge 0$. We begin by constructing an explicit barrier:

\begin{lemma}\label{upperbarrier}
{\rm (Barrier/2)} Let $\partial\Omega$ be $C^{1,1}$, $R\in \ ]0, \rho/4[$ and $\bar x\in D_{R/2}$. Then there exist $v\in\w\cap C(\R^N)$ and $C'_4=C'_4(N,p,s,\Omega)>1$ s.t.
\begin{enumroman}
\item\label{upperbarrier1} $\displaystyle |\fpl v|\le\frac{C'_4}{R^s}$ in $D_{2R}$;
\item\label{upperbarrier2} $v(\bar x)=0$;
\item\label{upperbarrier3} $\displaystyle v\ge\frac{\ds}{C'_4}$ in $D_R^c$;
\item\label{upperbarrier4} $|v|\le C'_4R^s$ in $D_{2R}$.
\end{enumroman}
\end{lemma}
\begin{proof}
We will construct the barrier as a solution of a double obstacle problem, and to this end we divide the proof in several steps.
\begin{figure}
\centering
\begin{tikzpicture}[scale=1.3]
\usetikzlibrary{patterns}
\filldraw (0, 0) circle (1pt);
\draw (3.5,3) node{$\Omega$};
\draw (-4,3.5) to [out=-20, in=180] (0, 0) to [out=0, in=200] (3.5,2) to [out=20, in=-100] (5, 3.5);
\draw[clip] (-4,3.5) to [out=-20, in=180] (0, 0) to [out=0, in=200] (3.5,2) to [out=20, in=-100] (5, 3.5);
\filldraw[lightgray] (0, 0) circle (2.6);
\filldraw[white] (0, 0) circle (1.6);
\draw (-4,3.5) to [out=-20, in=180] (0, 0) to [out=0, in=200] (3.5,2) to [out=20, in=-100] (5, 3.5);
\draw[clip] (0, 0) circle (3);
\filldraw[lightgray] (0,0) -- (-1.1,1.03) -- (-3.35, 3.69) -- (3.84, 3.184) -- (1.17,0.94) -- (0,0);
\filldraw[lightgray] (-1.1,1.03) circle (0.5);
\filldraw[lightgray] (1.17,0.94) circle (0.5);
\filldraw[lightgray](1.92,1.592) circle (0.5);
\filldraw[lightgray] (-1.675,1.845) circle (0.5);
\fill[pattern=dots, even odd rule] (0, 0) circle (2.4) (0, 0) circle (1.8);
\draw[very thin] (-1.1,1.03) circle (0.5);
\draw[very thin] (1.17,0.94) circle (0.5);
\draw[very thin] (1.92,1.592) circle (0.5);
\draw[very thin] (-1.675,1.845) circle (0.5);
\filldraw[white] (0, 0) circle (1);
\draw (0, 0) circle (3);
\draw (0, 0) circle (1);
\filldraw (0, 0) circle (1pt);
\draw (0, 2.6) node{$E_{R}$};
\draw[dashed] (0,0) -- node[above right]{$4R$} (-1, 2.83);
\draw[clip] (0, 0) circle (1);
\draw[dashed] (0,0) -- node[right]{$3R/4$} (0,1);
\draw[thick] (-4,3.5) to [out=-20, in=180] (0, 0) to [out=0, in=200] (3.5,2) to [out=20, in=-100] (5, 3.5);
\end{tikzpicture}
\caption{The regularized set $E_{R}$ in gray; in the dotted part we have ${\rm d}_\Omega\le C{\rm d}_{E_{R}}$.}
\label{fig3}
\end{figure}
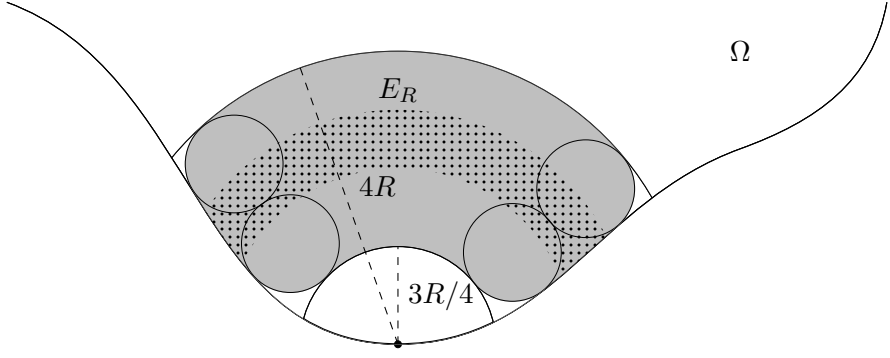

\noindent 
{\bf Step 1 (geometry).} Set
\[E_R=\bigcup\Big\{B_r(y):\,y\in\Omega,\,r\ge\frac{R}{8},\,B_r(y)\subset D_{4R}\setminus D_{3R/4}\Big\}.\]
By the regularity of $\partial\Omega$ stated in \eqref{grho} and $R<\rho/4$, $E_R\subset\Omega$ is a bounded domain with the interior sphere property with radius $\rho_{E_R}\ge R/16$ (see figure \ref{fig3}). We claim that
\beq\label{ubar3}
{\rm d}_\Omega\le C{\rm d}_{E_R} \quad \text{in $D_{3R}\setminus D_R$.}
\eeq
Indeed, fix a point $x\in D_{3R}\setminus D_R$. Since ${\rm d}_{E_R}(x)\ge R/8$ and
\[{\rm d}_{D_{7R/8}^c}(x)\le 3R+\frac{R}{4}\le 26\, {\rm d}_{E_R}(x),\]
we have ${\rm d}_{D_{7R/8}^c}(x)\le C{\rm d}_{E_R}(x)$.
By the triangle inequality and $R<\rho/4$ we have
\begin{align*}
{\rm d}_\Omega(x) &= |x-\Pi_\Omega(x)| \le |x-\Pi_\Omega(\Pi_{D_{7R/8}^c}(x))| \\
&\le |x-\Pi_{D_{7R/8}^c}(x)|+|\Pi_{D_{7R/8}^c}(x)-\Pi_\Omega(\Pi_{D_{7R/8}^c}(x))|.
\end{align*}
We distinguish two cases:
\begin{itemize}[leftmargin=0.6cm]
\item[$(a)$] if $\Pi_{D_{7R/8}^c}(x)\in\partial\Omega$, then
\[{\rm d}_\Omega(x) \le{\rm d}_{D_{7R/8}^c}(x)\le C{\rm d}_{E_R}(x);\]
\item[$(b)$] if $\Pi_{D_{7R/8}^c}(x)\notin\partial\Omega$, then $|\Pi_{D_{7R/8}^c}(x)|=7R/8$, which in turn implies $|\Pi_\Omega(\Pi_{D_{7R/8}^c}(x))|\le R$ and so
\[{\rm d}_\Omega(x)\le{\rm d}_{D_{7R/8}^c}(x)+R+\frac{7R}{8}\le C{\rm d}_{E_R}(x)+\frac{15R}{8}\le C{\rm d}_{E_R}(x).\]
\end{itemize}
In both cases we get \eqref{ubar3}. 
\vskip2pt
\noindent
{\bf Step 2 (lower obstacle).} Let $\tilde\varphi\in W^{s,p}_0(E_R)$ be the solution of the torsion problem
\[\begin{cases}
\fpl\tilde\varphi=1 & \text{in $E_R$} \\
\tilde\varphi=0 & \text{in $E_R^c$.}
\end{cases}\]
By Lemma \ref{subglobal} we have $\fpl\tilde\varphi\le 1$ in $\R^N$, while Lemma \ref{hopf} and the estimate on $\rho_{E_R}$ imply
\[\tilde\varphi\ge\frac{R^\frac{s}{p-2}}{C}{\rm d}_{E_R}^s \quad \text{in $\R^N$,}\]
with some $C>0$ depending on $N$, $p$, $s$. As in Section \ref{sec2}, we denote by $u_{4R}\in W^{s,p}_0(B_{4R})$ the solution to the torsion equation \eqref{torball} in $B_{4R}$. So, since $E_R\subset B_{4R}$, we have
\[\begin{cases}
\fpl\tilde\varphi\le\fpl u_{4R} & \text{in $B_{4R}$} \\
\tilde\varphi\le u_{4R} & \text{in $B_{4R}^c$.}
\end{cases}\]
By Proposition \ref{comp} and Lemma \ref{torest} we have
\[\tilde\varphi\le u_{4R}\le CR^\frac{s}{p-1}{\rm d}_{B_{4R}}^s\le CR^{p's} \quad \text{in $\R^N$.}\]
We set $\varphi=R^{-\frac{s}{p-1}}\tilde\varphi\in W^{s,p}_0(E_R)$, so by \cite[Lemma 2.9 $(i)$]{IMS1} and the inequalities above we have
\beq\label{ubar1}
\fpl\varphi=\frac{\fpl\tilde\varphi}{R^s}\le\frac{1}{R^s} \quad \text{in $\R^N$,}
\eeq
as well as
\beq\label{ubar2}
\varphi\le CR^{\big(p'-\frac{1}{p-1}\big)s}=CR^s \quad \text{in $\R^N$.}
\eeq
 Now, by \eqref{ubar3} and Lemma \ref{hopf} we have
\beq\label{ubar4}
\varphi \ge \frac{{\rm d}_{E_R}^s}{C} \ge \frac{\ds}{C} \quad \text{in $D_{3R}\setminus D_R$.}
\eeq
\vskip2pt
\noindent
{\bf Step 3 (upper obstacle).} Pick $\lambda>0$ (to be determined later) and set for all $x\in\R^N$
\[\psi(x)=\frac{\lambda}{R^\frac{s}{p-1}}\big(\max_{\R^N}u_{R/8}-u_{R/8}(x-\bar x)\big),\]
where $u_{R/8}\in W^{s,p}_0(B_{R/8})$ solves \eqref{torball} in $B_{R/8}$. Clearly $\psi\in \widetilde{W}^{s,p}(\Omega)$, $\psi\ge 0$ and $\psi(\bar x)=0$ (since $u_{R/8}$ is radially decreasing in $B_{R/8}$). We claim that for all $\lambda(N, p, s, \Omega)>0$ big enough 
\beq\label{ubar6}
\psi\ge \varphi \quad \text{in $\R^{N}$.}
\eeq
Indeed, fix $x\in\R^N$. Two cases may occur:
\begin{itemize}[leftmargin=0.6cm]
\item[$(a)$] if $x\in D_{3R/4}$, then $\varphi(x)=0$, while $\psi(x)\ge 0$;
\item[$(b)$] if $x\in  D^{c}_{3R/4}$, then $|x-\bar x|>R/8$, hence $u_{R/8}(x-\bar x)=0$, while by Lemma \ref{torball}
\[\max_{\R^N}u_{R/8} \ge \frac{R^\frac{s}{p-1}}{C}\max_{\R^N}{\rm d}_{B_{R/8}}^s\ge\frac{R^{p's}}{C},\]
which in turn implies $\psi(x)\ge\lambda R^s/C$.  By using \eqref{ubar2}, we have $\varphi(x)\le \psi(x)$ for large enough $\lambda$.
\end{itemize}
In both cases we have \eqref{ubar6} for some  $\lambda(N, p, s, \Omega)>0$ which will be fixed henceforth. By \cite[Lemma 2.9 $(i)$]{IMS1} and Lemma \ref{subglobal} we have
\beq\label{ubar7}
\fpl\psi\ge -\frac{C}{R^s} \quad \text{in $\R^{N}$.}
\eeq
One last property of $\psi$ is that
\beq\label{ubar8}
\psi\le CR^s \quad \text{in $\R^{N}$,}
\eeq
which follows from the upper bound in Lemma \ref{torball} and 
\[\psi\le\frac{C}{R^\frac{s}{p-1}}\max_{\R^{N}} u_{R/8}\le C\max_{\R^{N}}{\rm d}_{B_{R/8}}^s \le CR^s.\]
\vskip2pt
\noindent
{\bf Step 4 (the barrier).} Consider the constrained minimization problem
\beq\label{ubar9}
\min\big\{[u]_{s,p}^p:\,u\in\w, \quad \varphi\le u\le\psi\,\text{ in $\R^N$}\big\}.
\eeq
By Lemma \ref{dobstacle}, problem \eqref{ubar9} has a solution $\tilde v\in\w$, which satisfies
\[0\wedge\fpl\psi\le\fpl\tilde v\le 0\vee\fpl\varphi \quad \text{in $\Omega$}.\]
By \eqref{ubar1}, \eqref{ubar7} we have
\beq\label{ubar10}
|\fpl\tilde v|\le\frac{C}{R^s} \quad \text{in $D_{2R}$.}
\eeq
Besides, since $\varphi(\bar x)=\psi(\bar x)=0$ we deduce $\tilde v(\bar x)=0$, while \eqref{ubar8} implies
\beq\label{ubar11}
0\le\tilde v\le CR^s \quad \text{in $\R^{N}$.}
\eeq
Moreover, by \eqref{ubar4} we have
\beq\label{ubar12}
\tilde v\ge\frac{\ds}{\tilde C} \quad \text{in $D_{3R}\setminus D_R$,}
\eeq
for some $\tilde C=\tilde C(N,p,s,\Omega)>0$. Still, $\tilde v$ is not the desired function as it only satisfies the lower bound \eqref{ubar12} in $D_{3R}\setminus D_R$. So we need to extend \eqref{ubar12} to the larger set $D_R^c$ while keeping the other properties. Set for all $x\in\R^N$
\[v(x)=\begin{cases}
\tilde v(x) & \text{if $x\in D_{3R}$} \\
\displaystyle\tilde v(x)\vee\frac{\ds(x)}{\tilde C} & \text{if $x\in D_{3R}^c$.}
\end{cases}\]
Clearly $v\in\w$ satisfies \ref{upperbarrier2} and \ref{upperbarrier4},  since, by \eqref{ubar12}, we are changing $\tilde v$ only outside of $D_{3R}$. Moreover, \ref{upperbarrier3} now holds by construction. So, it remains to check \ref{upperbarrier1} for $v$.
By Proposition \ref{spp} we have for all $x\in D_{2R}$
\beq\label{ubar13}
\fpl v(x)=\fpl\tilde v(x)+2\int_{D_{3R}^c\cap\{\tilde v<\ds/\tilde C\}}\frac{(\tilde v(x)-\ds(y)/\tilde C)^{p-1}-(\tilde v(x)-\tilde v(y))^{p-1}}{|x-y|^{N+ps}}\,dy.
\eeq
By the monotonicity of $t\mapsto t^{p-1}$ the integrand is negative and \eqref{ubar10} yelds
\[\fpl v\le\frac{C'_4}{R^s}\quad \text{in $D_{2R}$}.\]
On the other hand, for all $x\in D_{2R}$, $y\in D_{3R}^c$ we have  by \eqref{ubar11}
\[\Big|\tilde v(x)-\frac{\ds(y)}{\tilde C}\Big| \le C (R^{s}+|y|^{s}) \le C|y|^{s}\]
and
\[|\tilde v(x)-\tilde v(y)| \le CR^{s} \le C|y|^{s}.\]
Since $|x-y|\ge |y|/3$ for all $x\in D_{2R}$, $y\in D_{3R}^c$,  plugging these inequalities into \eqref{ubar13} gives
\[\fpl v \ge -\frac{C'_{4}}{R^s}-C\int_{D_{3R}^c}\frac{dy}{|y|^{N+s}}\,dy \ge -\frac{C'_4}{R^s}\quad \text{in $D_{2R}$},\]
for a possibly larger $C'_{4}>1$ (depending on $N$, $p$, $s$, and $\Omega$), which concludes the proof of \ref{upperbarrier1}.
\end{proof}

\noindent
The next result shows that, if a subsolution of \eqref{sub} is small enough in $\tilde B_R$, then it is actually negative in $D_{R/2}$:

\begin{lemma}\label{negative}
Let $\partial\Omega$ be $C^{1,1}$, $R\in \ ]0, \rho/4[$, $p\ge 2$ and $u\in\widetilde{W}^{s,p}(D_R)$ satisfy \eqref{sub}. Then there exists $C'_5=C'_5(N,p,s,\Omega)>1$ s.t.\
\[\Ex(u, M, R)\ge C'_5\big(M+(KR^s)^\frac{1}{p-1}+HR^s\big) \ \Longrightarrow \ \sup_{D_{R/2}}\,u\le 0.\]
\end{lemma}
\begin{proof}
Fix $\bar x\in D_{R/2}$, and let $v\in\w$ be the barrier in the previous Lemma. Set
\[w(x)=\begin{cases}
C'_4Mv(x) & \text{if $x\in\tilde B_R^c$} \\
u(x) & \text{if $x\in\tilde B_R$,}
\end{cases}\]
 for all $x\in\R^N$, $C'_4>1$ being as in Lemma \ref{upperbarrier}.
Recall that ${\rm dist}\,(D_R,\tilde B_R)>0$. By Proposition \ref{spp}, \cite[Lemma 2.9 $(i)$]{IMS1}, inequality \eqref{ab}, and Lemma \ref{upperbarrier} \ref{upperbarrier1} \ref{upperbarrier3}, for all $x\in D_R$ we have
\[\begin{split}
 \fpl w(x)&= \fpl (C'_4Mv(x))+2\int_{\tilde B_r}\frac{(C'_4Mv(x)-u(y))^{p-1}-(C'_4Mv(x)-C'_4Mv(y))^{p-1}}{|x-y|^{N+ps}}\,dy \\
&\ge (C'_4M)^{p-1}\fpl v(x)+\frac{1}{C}\int_{\tilde B_R}\frac{(C'_4Mv(y)-u(y))^{p-1}}{|x-y|^{N+ps}} \\
&\ge -\frac{CM^{p-1}}{R^s}+\frac{1}{CR^{ps}}\dashint_{\tilde B_R}(M\ds(y)-u(y))^{p-1}\,dy.
\end{split}\]
By the properties \eqref{bt} of $\tilde B_{R}$, H\"older's inequality (recall that $p\ge 2$), and $u\le M\ds$ in $\tilde B_{R}$
\[\dashint_{\tilde B_R}(M\ds(y)-u(y))^{p-1}\,dy\ge \frac{R^{s(p-1)}}{C}\dashint_{\tilde B_R}\Big(M-\frac{u(y)}{\ds (y)}\Big)^{p-1}\,dy\ge \frac{R^{s(p-1)}}{C}\Ex(u, M, R)^{p-1},\]
so that 
\beq\label{neg2}
\fpl w\ge -\frac{CM^{p-1}}{R^s}+\frac{\Ex(u, M, R)^{p-1}}{C R^{s}} \quad \text{in $D_{R}$},
\eeq
for some $C\ge C'_4$. Now set
\[C'_5=(3C^2)^\frac{1}{p-1}\ge (3C)^\frac{1}{p-1},\]
which only depends on $N$, $p$, $s$, and $\Omega$. Assume
\beq\label{neg3}
\Ex(u, M, R)\ge C'_5\big(M+(KR^s)^\frac{1}{p-1}+HR^s\big).
\eeq
A straightforward computation leads from \eqref{neg3} to the following inequalities
\[\Ex(u, M, R)^{p-1} \ge
\begin{cases}
(C'_5M)^{p-1}\ge 3C^2M^{p-1} \\
(C'_5)^{p-1}KR^s\ge 3CKR^s \\
(C'_5M)^{p-2}\Ex(u, M, R)\ge 3CM^{p-2}HR^s,
\end{cases}\]
and hence to
\[\Ex(u, M, R)^{p-1}\ge C^2M^{p-1}+CKR^s+CM^{p-2}HR^s.\]
So, by \eqref{neg2} we have
\[\fpl w\ge K+M^{p-2}H\ge \fpl u\quad \text{in $D_{R}$}.\]
Besides, we have $ u\le w$ in $D_{R}^{c}$: indeed, if $x\in\tilde B_R$ there is nothing to prove. If $x\in D_R^c\setminus\tilde B_R$, by \eqref{sub} and Lemma \ref{upperbarrier} \ref{upperbarrier3} we have
\[u(x)\le M\ds(x)\le C'_4Mv(x)=w(x).\]
Summarizing, we obtained
\[\begin{cases}
\fpl u\le\fpl w & \text{in $D_R$} \\
u\le w & \text{in $D_R^c$.}
\end{cases}\]
By Proposition \ref{comp} we have $u\le w$ in $\R^N$. In particular, by Lemma \ref{upperbarrier} \ref{upperbarrier2} we get $u(\bar x)\le 0$. By arbitrariness of $\bar x\in D_{R/2}$, the proof is concluded.
\end{proof}

\noindent
Now we can prove our upper bounds on subsolutions. First we prove an upper bound for large values $\mathcal{L}_{p-1}$. 
\begin{lemma}\label{upperbig}
Let $\partial\Omega$ be $C^{1,1}$, $p\ge 2$ and $u\in\widetilde{W}^{s,p}(D_R)$ satisfy \eqref{sub}. Then there exist $\theta'_1=\theta'_1(N,p,s,\Omega)\ge 1$, $\sigma'_1=\sigma'_1(N,p,s,\Omega)\in\ ]0,1]$, and $C'_6=C'_6(N,p,s,\Omega)>1$ s.t.\ for all $R\in \ ]0, \rho/4[$
\[\Ex(u, M, R)\ge M\theta'_1 \ \Longrightarrow \ \inf_{D_{R/4}}\Big(M-\frac{u}{\ds}\Big)\ge \sigma'_1\Ex(u, M, R)-C'_6(KR^s)^\frac{1}{p-1}-C'_6HR^s.\]
\end{lemma}
\begin{proof}
We set
\[H_R=\bigcup\Big\{B_r(y):\,y\in D_{3R/8},\,r\ge\frac{R}{16},\,B_r(y)\subset D_{3R/8}\Big\}.\]
By \eqref{grho}, $H_R$ satisfies the interior sphere property with radius $\rho_{H_R}\ge R/32$. Moreover,
\beq\label{ub1}
{\rm d}_\Omega\le C{\rm d}_{H_R} \quad \text{in $D_{R/4}$}
\eeq
for some $C>1$, (this is proved exactly as \eqref{lb1}, changing the radii). Let $\varphi\in W^{s,p}_0(H_R)$ solve \beq\label{ub2}
\begin{cases}
\fpl\varphi=1 & \text{in $H_R$} \\
\varphi=0 & \text{in $H_R^c$.}
\end{cases}
\eeq
By Lemma \ref{subglobal} we have $\fpl\varphi\le 1$ in $\R^N$. Besides we have
\beq\label{ub3}
\frac{R^\frac{s}{p-1}}{C}\ds\le\varphi\le CR^{p's} \quad \text{in $D_{R/4}$,}
\eeq
the first inequality coming from Lemma \ref{hopf} and \eqref{ub1}, while the second is proved as in Lemma \ref{upperbarrier} by comparing $\varphi$ to $u_{R/2}$. Now pick $\lambda>0$ (to be  determined later) and set for all $x\in\R^N$
\[v(x)=\begin{cases}
\displaystyle -\frac{\lambda}{R^\frac{s}{p-1}}\varphi(x) & \text{if $x\in D_{R/2}$} \\
M\ds(x) & \text{if $x\in D_{R/2}^c.$}
\end{cases}\]
Clearly $v\in\widetilde{W}^{s,p}(H_R)$ and ${\rm dist}\,(D_{R/2}^c,\,H_R)>0$. So we can apply Proposition \ref{spp}, which along with \cite[Lemma 2.9 $(i)$]{IMS1}, \eqref{ub2}  and some direct calculations  yields for all $x\in H_R\subset D_{R/2}$
\[
\begin{split}
\fpl v(x) &= -\frac{\lambda^{p-1}}{R^s}\fpl\varphi(x)+2\int_{D_{R/2}^c}\frac{(-\lambda R^{-\frac{s}{p-1}}\varphi(x)-M\ds(y))^{p-1}-(-\lambda R^{-\frac{s}{p-1}}\varphi(x))^{p-1}}{|x-y|^{N+ps}}\,dy \\
&\ge -\frac{\lambda^{p-1}}{R^s}-C\int_{D_{R/2}^c}\frac{\lambda^{p-1}R^{-s}\varphi^{p-1}(x)+M^{p-1}{\rm d}_\Omega^{(p-1)s}(y)}{|x-y|^{N+ps}}\,dy. 
\end{split}
\]
Therefore, using $C|y-x|>|y|$ for $x\in H_{R}$ and $y\in B^{c}_{R/2}$, \eqref{ub3} and ${\rm d}_{\Omega}(y)\le |y|$
\beq 
\label{ub4}
\begin{split}
\fpl v(x)&\ge -\frac{\lambda^{p-1}}{R^s}-C(\lambda^{p-1}+M^{p-1})\int_{B_{R/2}^c}\frac{R^{(p-1)s}+|y|^{(p-1)s}}{|y|^{N+ps}}\,dy\\
&\ge -C\,\frac{\lambda^{p-1}+M^{p-1}}{R^s},
\end{split}
\eeq
for $x\in H_{R}$. Further, set for all $x\in \R^{N}$
\[w(x)=\begin{cases}
v(x) & \text{if $x\in\tilde B_R^c$} \\
u(x) & \text{if $x\in\tilde B_R$,}
\end{cases}\]
where $\tilde B_R$ is defined in \eqref{bt}. By Proposition \ref{spp}, $w\in\widetilde{W}^{s,p}(H_R)$ and for all $x\in H_R$
\beq\label{ub5}
\begin{split}
\fpl w(x) &= \fpl v(x)+2\int_{\tilde B_R}\frac{(v(x)-u(y))^{p-1}-(v(x)-M\ds(y))^{p-1}}{|x-y|^{N+ps}}\,dy \\
 &\ge -C\,\frac{\lambda^{p-1}+M^{p-1}}{R^s}+\frac{1}{C}\int_{\tilde B_R}\frac{(M\ds(y)-u(y))^{p-1}}{|x-y|^{N+ps}}\,dy \\
 &\ge -C\,\frac{\lambda^{p-1}+M^{p-1}}{R^s}+\frac{1}{CR^s}\dashint_{\tilde B_R}\Big(M-\frac{u(y)}{\ds(y)}\Big)^{p-1}\,dy \\
&\ge -C\,\frac{\lambda^{p-1}+M^{p-1}}{R^s}+\frac{\Ex(u, M, R)^{p-1}}{CR^s},
\end{split}
\eeq
where we have also used \eqref{ub4}, \eqref{ab} and H\"older's inequality. So far, $C>1$ has been chosen as big as necessary to satisfy all inequalities, depending only on $N$, $p$, $s$, and $\Omega$. Now we can fix the constants in such a way that either the thesis is trivial or $w$ is an upper barrier for $u$.  Set
\begin{align*}
&\lambda=\frac{\Ex(u, M, R)}{(4C^2)^\frac{1}{p-1}}, & &  \theta'_1=\max\big\{2C'_5,\,(4C^2)^\frac{1}{p-1}\big\},\\
&\sigma'_1=\frac{1}{C(4C^2)^\frac{1}{p-1}}, & & C'_6=\sigma'_1\max\Big\{2C'_5,\,(4C)^\frac{1}{p-1},\,\frac{4C}{(\theta'_1)^{p-2}}\Big\},
\end{align*}
where $C'_5>0$ is as in Lemma \ref{negative}. Clearly $C'_6>1$, and all these constants (except $\lambda$) only depend on $N$, $p$, $s$, and $\Omega$. Now we prove the asserted implication. Assume
\beq\label{ub6}
\Ex(u, M, R)\ge M\theta'_1.
\eeq
Then, with the previous choices, \eqref{ub5} implies
\beq
\label{ub5b}
\begin{split}
\fpl w&\ge  \frac{C}{R^{s}}\left[- \frac{\Ex(u, M, R)^{p-1}}{4C^{2}}   -\Big(\frac{\Ex(u, M, R)}{\theta_{1}'}\Big)^{p-1}+\frac{\Ex(u, M, R)^{p-1}}{C^{2}}\right]\\
& \ge\frac{\Ex(u, M, R)^{p-1}}{2CR^s}.
\end{split}
\eeq
We can also assume
\beq\label{ub7}
\sigma'_1\Ex(u, M, R)\ge C'_6(KR^s)^\frac{1}{p-1}+C'_6HR^s,
\eeq
otherwise there is nothing to prove (recall that $u$ satisfies \eqref{sub}). Such relation and \eqref{ub6} imply
\[\Ex(u, M, R)^{p-1} \ge\begin{cases}
\displaystyle\Big(\frac{C'_6}{\sigma'_1}\Big)^{p-1}KR^s \ge 4CKR^s \\[2pt]
\displaystyle(M\theta'_1)^{p-2}\frac{C'_6}{\sigma'_1}HR^s \ge 4CM^{p-2}HR^s,
\end{cases}\]
and in turn
\[\frac{\Ex(u, M, R)^{p-1}}{2CR^s}\ge K+M^{p-2}H.\]
Plugging the last inequality  into \eqref{ub5b}, we get 
\beq\label{ub8}
\fpl w\ge K+M^{p-2}H\ge \fpl u\quad \text{in $H_{R}$}.
\eeq
Let us now consider the pointwise estimates  for $x\in H_{R}^{c}$. Three cases may occur:
\begin{itemize}[leftmargin=0.6cm]
\item[$(a)$] if $x\in\tilde B_R$, then $w(x)=u(x)$;
\item[$(b)$] if $x\in D_{R/2}^c\cap\tilde B_R^c$, then $w(x)=M\ds(x)\ge u(x)$ by assumption;
\item[$(c)$] if $x\in D_{R/2}\cap H_R^c$, by \eqref{ub7}, \eqref{ub6} we also have
\[\Ex(u, M, R)\ge\begin{cases}
\displaystyle\frac{C'_6}{\sigma'_1}(KR^s)^\frac{1}{p-1}+\frac{C'_6}{\sigma'_1}HR^s \ge 2C'_5(KR^s)^\frac{1}{p-1}+2C'_5HR^s \\
M\theta'_1 \ge 2C'_5M,
\end{cases}\]
which summarizes as
\[\Ex(u, M, R)\ge C'_5\big(M+(KR^s)^\frac{1}{p-1}+HR^s\big),\]
and by Lemma \ref{negative} implies $u\le 0$ in $D_{R/2}$, hence $w(x)=0\ge u(x)$.
\end{itemize}
Therefore $w\le u$ in $H_{R}^{c}$, and recalling \eqref{ub8} we therefore have
\[\begin{cases}
\fpl u\le\fpl w & \text{in $H_R$} \\
u\le w & \text{in $H_R^c$.}
\end{cases}\]
By Proposition \ref{comp} we deduce $u\le w$ in $\R^N$. In particular, for all $x\in D_{R/4}$ we have (recalling the definitions of $\varphi$, $v$, $w$, and of $\lambda$)
\[u(x) \le -\frac{\lambda\varphi(x)}{R^\frac{s}{p-1}} \le -\sigma'_1\Ex(u, M, R)\ds(x).\]
So we have
\[\inf_{D_{R/4}}\Big(M-\frac{u}{\ds}\Big) \ge -\sup_{D_{R/4}}\frac{u}{\ds} \ge \sigma'_1\Ex(u, M, R),\]
which readily yields the conclusion.
\end{proof}

\noindent
Now we prove a similar upper bound for the case when $\Ex(u, M, R)$ is small:

\begin{lemma}\label{uppersmall}
Let $\partial\Omega$ be $C^{1,1}$, $p\ge 2$, $u\in\widetilde{W}^{s,p}(D_R)$ solve \eqref{sub} and $R\in \ ]0, \rho/4[$. Then, for all $\theta\ge 1$ there exist $\sigma'_\theta=\sigma'_\theta(N,p,s,\Omega,\theta)\in\ ]0,1]$, $C'_\theta=C'_\theta(N,p,s,\Omega,\theta)>1$ s.t.\
\[\Ex(u, M, R)\le M\theta \ \Longrightarrow \ \inf_{D_{R/2}}\Big(M-\frac{u}{\ds}\Big)\ge\sigma'_\theta\Ex(u, M, R)-C'_\theta(M^{p-1}+K)^\frac{1}{p-1}R^\frac{s}{p-1}-C'_\theta HR^s.\]
\end{lemma}
\begin{proof}
The proof is similar to the one of Lemma \ref{lowersmall} and we only sketch it. Fix $\varphi\in C^\infty_c(B_1)$ s.t.\ $\varphi=1$ in $B_{1/2}$ and $0\le\varphi\le 1$ in $B_1$, let $\lambda_1>0$ be as in Lemma \ref{lowerbarrier}, and for all $\lambda\in\ ]0,\lambda_1]$ set
\[w_\lambda(x)=M\Big(1-\lambda\varphi\Big(\frac{x}{R}\Big)\Big)\ds(x), \qquad  x\in\R^N.\]
Without loss of generality we may assume $\lambda_1\le 1$. Then $w_\lambda\in W^{s,p}(D_R)$ and it satisfies
\[\begin{cases}
\displaystyle\fpl w_\lambda\ge -C_6M^{p-1}\Big(1-\frac{\lambda}{R^s}\Big) & \text{in $D_R$} \\
w_\lambda= M(1-\lambda)\ds & \text{in $D_{R/2}$}
\end{cases}\]
($C_6>0$ as in Lemma \ref{lowerbarrier}). Now set for all $x\in\R^N$
\[v_\lambda(x)=\begin{cases}
w_\lambda(x) & \text{if $x\in\tilde B_R^c$} \\
u(x) & \text{if $x\in\tilde B_R$,}
\end{cases}\]
where $\tilde B_R$ is defined as in \eqref{bt}. By Proposition \ref{spp}, we have for all $x\in D_R$
\[\fpl v_\lambda(x) = \fpl w_\lambda(x)+2\int_{\tilde B_R}\frac{(w_\lambda(x)-u(y))^{p-1}-(w_\lambda(x)-w_\lambda(y))^{p-1}}{|x-y|^{N+ps}}\,dy,\]
and estimating the integral term as in the proof of Lemma \ref{lowersmall}, we obtain
\beq\label{us1}
\fpl v_{\lambda}\ge -CM^{p-1}-\frac{M^{p-2}}{R^s}\Big(CM\lambda-\frac{\Ex(u, M, R)}{C}\Big),
\eeq
for some $C>1$ (depending on $N$, $p$, $s$, and $\Omega$). Now we fix $\theta\ge 1$ and set
\[\sigma'_\theta=\frac{\lambda_1}{2\theta C^2}, \quad C'_\theta=\sigma'_\theta\max\big\{4C,\,(4C^2\theta^{p-2})^\frac{1}{p-1}\big\}, \quad \lambda=\frac{\sigma'_\theta\Ex(u, M, R)}{M}.\]
Note that $\sigma'_\theta\le 1$. We also assume
\beq\label{us2}
\Ex(u, M, R)\le M\theta.
\eeq
Then, by the choice of constants we have
\[\lambda<\lambda_1, \quad CM\lambda\le\frac{\Ex(u, M, R)}{2C}.\]
These inequalities and \eqref{us1} give
\[\fpl v_\lambda\ge -CM^{p-1}+\frac{M^{p-2}}{R^{s}}\frac{\Ex(u, M, R)}{2C}\quad \text{in $D_{R}$}.\] 
Assuming also  
\[\sigma'_\theta\Ex(u, M, R)\ge C'_\theta(M^{p-1}+K)^\frac{1}{p-1}R^\frac{s}{p-1}+C'_\theta HR^s\]
(otherwise the thesis is trivial), the choice of the parameters and \eqref{us2} imply
\[M^{p-2}\Ex(u,M,R)\ge 2C(CM^{p-1}+K+M^{p-2}H)R^s,\]
exactly as in the proof of Lemma \ref{lowersmall}, and therefore
\[\fpl v_{\lambda}\ge K+M^{p-2}H \quad \text{in $D_{R}$.}\]
Moreover in $D_{R}^{c}$ we have by construction either $v_{\lambda}=u$ in $\tilde{B}_{R}$, or $v_{\lambda}=w_{\lambda}=M\ds\ge u$. Thus 
\[\begin{cases}
\fpl u\le\fpl v_\lambda & \text{in $D_R$} \\
u\le v_\lambda & \text{in $D_R^c$.}
\end{cases}\]
Proposition \ref{comp} ensures  $u\le v_\lambda$ in all of $\R^N$. In particular $u\le w_\lambda=M(1-\lambda)\ds$ in $D_{R/2}$. So, 
\[\inf_{D_{R/2}}\Big(M-\frac{u}{\ds}\Big) \ge \inf_{D_{R/2}}\Big(M-\frac{w_\lambda}{\ds}\Big) \ge M\lambda = \sigma'_\theta\Ex(u, M, R)\]
and the conclusion follows.
\end{proof}

\noindent
Now we present the analog of Proposition \ref{lower}, dealing with the problem
\beq\label{subloc}
\begin{cases}
\fpl u\le \tilde K & \text{in $D_R$} \\
u\le M\ds & \text{in $D_{2R}$,}
\end{cases}
\eeq
with $\tilde K,M\ge 0$.

\begin{proposition}\label{upper}
{\rm (Upper bound)} Let $\partial\Omega$ be $C^{1,1}$, $p\ge 2$, $u\in\widetilde{W}^{s,p}_0(D_R)$ solve \eqref{subloc} and $R\in \ ]0, \rho/4[$. Then, for all $\eps>0$ there exist $\tilde C'_\eps=\tilde C'_\eps(N,p,s,\Omega,\eps)>0$ and two more constants $\sigma'_2=\sigma'_2(N,p,s,\Omega)\in(0,1]$, $C'_7=C'_7(N,p,s,\Omega)>1$ s.t.\ 
\begin{align*}
\inf_{D_{R/4}}\Big(M-\frac{u}{\ds}\Big) &\ge \sigma'_2\Ex(u, M, R)-\eps\Big\|M-\frac{u}{\ds}\Big\|_{L^\infty(D_R)}  - C'_7{\rm tail}_1\Big(\Big(\frac{u}{\ds}-M\Big)_+,2R\Big)R^s\\
&\quad -\tilde C'_\eps\Big[M+\tilde K^\frac{1}{p-1}+{\rm tail}_{p-1}\Big(\Big(\frac{u}{\ds}-M\Big)_+,2R\Big)\Big]R^\frac{s}{p-1}.
\end{align*}
\end{proposition}
\begin{proof}
The proof is identical to the one of Proposition \ref{lower}, so we only sketch it. Consider $v=u\wedge M\ds$ and fix $\eps>0$. By Proposition \ref{updown} \ref{updown2} 
\[\begin{cases}
\fpl v \le K+M^{p-2}H \quad \text{in $D_{R}$} \\
v\le M\ds \quad\text{in $\R^{N}$},
\end{cases}\]
where 
\[K:=\tilde K+\frac{\eps^{p-1}}{R^s}\Big\|M-\frac{u}{\ds}\Big\|_{L^\infty(D_R)}^{p-1}+C'_\eps{\rm tail}_{p-1}\Big(\Big(\frac{u}{\ds}-M\Big)_+,2R\Big)^{p-1},\]
\[H:=C'_3{\rm tail}_1\Big(\Big(\frac{u}{\ds}-M\Big)_+,2R\Big).\]
Let $0<\sigma'_1\le 1\le\theta'_1,C'_6$ given in  Lemma \ref{upperbig} and choose $\theta=\theta_{1}'$ in Lemma \ref{uppersmall}, with corresponding $0<\sigma'_{\theta'_1}\le 1\le C'_{\theta'_1}$ given therein. Define
\[\sigma'_2=\min\{\sigma'_1,\sigma'_{\theta'_1}\}, \quad C=\max\{C'_6,C'_{\theta'_1}\}.\]
Considering separately the cases $\Ex(u, M, R)\ge M\theta_{1}$ and $\Ex(u, M, R)<M\theta_{1}$   we obtain that 
\[\inf_{D_{R/4}}\Big(M-\frac{v}{\ds}\Big) \ge \sigma'_2\Ex(v, M, R)-C(M^{p-1}+K)^\frac{1}{p-1}R^\frac{s}{p-1}-CHR^s.\]
Since $u=v$ in $D_{2R}$, after standard estimates we conclude.
\end{proof}

\section{Weighted H\"older regularity}\label{sec5}

\noindent
This final section is devoted to the proof of Theorem \ref{main}, i.e., of weighted H\"older regularity for the solutions of problem \eqref{dp}. We follow a standard approach, starting with an estimate of the oscillation near the boundary of $u/\ds$, where $u$ satisfy
\beq\label{bound}
\begin{cases}
|\fpl u|\le K & \text{in $\Omega$} \\
u=0 & \text{in $\Omega^c$,}
\end{cases}
\eeq
with some $K>0$. Our estimate reads as follows:

\begin{theorem}\label{osc}
Let $\partial\Omega$ be $C^{1,1}$, $p\ge 2$, $x_1\in\partial\Omega$ and $u\in\w$ solve \eqref{bound}. Then there exist $\alpha_1\in\ ]0,s]$, $C_8>1$, $R_0=R_0\in\ ]0,\rho/4[$ all depending on $N, p, s$ and $\Omega$ s.t.\ for all $r\in\ ]0,R_0[$
\[\underset{D_r(x_1)}{\rm osc}\,\frac{u}{\ds}\le C_8K^\frac{1}{p-1}r^{\alpha_1}.\]
\end{theorem}
\begin{proof}
First we assume $x_1=0$ and $K=1$ in \eqref{bound}. We set $v=u/\ds\in\widetilde{W}^{s,p}_0(\Omega)$, $R_0=\min\{1,\rho/4\}$, and for all $n\in\N$ we define $R_n=R_0/8^n$, $D_n=D_{R_n}$, and $\tilde B_n=\tilde B_{R_n/2}$ (see \eqref{bt}). We claim that there exist $\alpha_1\in\ ]0,s]$, $\mu\ge 1$, a nondecreasing sequence $\{m_n\}$, and a nonincreasing sequence $\{M_n\}$ in $\R$ (all depending on $N$, $p$, $s$, and $\Omega$) s.t.\ for all $n\in\N$
\beq\label{osc1}
m_n\le\inf_{D_n}v\le\sup_{D_n}v\le M_n, \quad M_n-m_n=\mu R_n^{\alpha_1}.
\eeq
Pick $\alpha_1\in\ ]0,s]$ (to be determined later). We argue by (strong) induction on $n\in\N$. The first step $n=0$ follows from \cite[Theorem 4.4]{IMS1}, which (slightly rephrased) ensures existence of $\tilde C_\Omega>1$ (depending on $N$, $p$, $s$, and $\Omega$) s.t.\
\[|v|\le \tilde C_\Omega \quad \text{in $\Omega$.}\]
So we set $M_0=\tilde C_\Omega$, $m_n=-\tilde C_\Omega$, $\mu=2\tilde C_\Omega/R_0^{\alpha_1}$, and \eqref{osc1} holds. Now let $n\in\N$ and
\[m_0\le\ldots\le m_n<M_n\le\ldots\le M_0\]
be defined and satisfy \eqref{osc1}. We set $R=R_n/2$, so $D_{n+1}=D_{R/4}$ and $\tilde B_R=\tilde B_n$, and aim at applying our lower and upper bounds on $v$, by distinguishing three cases:
\begin{itemize}[leftmargin=0.6cm]
\item[$(a)$] If $0\le m_n<M_n$, then $u$ satisfies both \eqref{superloc} and \eqref{subloc} with $\tilde K=1$ and non-negative multipliers of $\ds$, namely
\[\begin{cases}
-1\le\fpl u\le 1 & \text{in $D_{R_n/2}$} \\
m_n\ds\le u\le M_n\ds & \text{in $D_n$.}
\end{cases}\]
Thus, Propositions \ref{lower} and \ref{upper} apply, yielding constants $0<\sigma\le 1<C_7,C_\eps$ (we take here the smaller of $\sigma$'s and the bigger of $C_7$'s and of $C_\eps$'s, all depending on $N$, $p$, $s$, $\Omega$ with $C_\eps$ also depending on $\eps$) s.t.\ the following bounds hold:
\beq\label{osc2}
\begin{split}
\inf_{D_{n+1}}(v-m_n) &\ge \sigma\dashint_{\tilde B_n}(v-m_n)\,dx-C_\eps\big[m_n+1+{\rm tail}_{p-1}((m_n-v)_+,R_n)\big]R_n^\frac{s}{p-1} \\
&\quad - C_7\Big[\eps\sup_{D_{R_n/2}}(v-m_n)+{\rm tail}_1((m_n-v)_+,R_n)R_n^s\Big],
\end{split}
\eeq
\beq\label{osc3}
\begin{split}
\inf_{D_{n+1}}(M_n-v) &\ge \sigma\dashint_{\tilde B_n}(M_n-v)\,dx-C_\eps\big[M_n+1+{\rm tail}_{p-1}((v-M_n)_+,R_n)\big]R_n^\frac{s}{p-1} \\
&\quad - C_7\Big[\eps\sup_{D_{R_n/2}}(M_n-v)+{\rm tail}_1((v-M_n)_+,R_n)R_n^s\Big].
\end{split}
\eeq
\item[$(b)$] If $m_n<0<M_n$, then we can similarly apply Proposition \ref{upper} to $u$ with upper bound $M_{n}\ds$ and to $-u$ with upper bound $-m_{n}\ds$. After substitution, this provides \eqref{osc3} and  \eqref{osc2} respectively.
\item[$(c)$] If $m_n<M_n\le 0$, then we apply Proposition \ref{lower} to $-u$ with lower bound $-M_{n}\ds$ and Proposition \ref{upper} to $-u$ with upper bound $-m_{n}\ds$, getting again \eqref{osc3} and \eqref{osc2} respectively.
\end{itemize}
All in all, by taking convenient constants and replacing $\eps$ with $\eps/C_7$, we have
\begin{align*}
\sigma(M_n-m_n) &= \sigma\dashint_{\tilde B_n}(M_n-v)\,dx+\sigma\dashint_{\tilde B_n}(v-m_n)\,dx \\
&\le \inf_{D_{n+1}}(M_n-v)+\inf_{D_{n+1}}(v-m_n)+\eps\sup_{D_n}(M_n-v)+\eps\sup_{D_n}(v-m_n) \\
&\quad + C_\eps\big[1+|M_n|+|m_n|+{\rm tail}_{p-1}((v-M_n)_+,R_n)+{\rm tail}_{p-1}((m_n-v)_+,R_n)\big]R_n^\frac{s}{p-1} \\[2pt]
&\quad + C_7\big[{\rm tail}_1((v-M_n)_+,R_n)+{\rm tail}_1((m_n-v)_+,R_n)\big]R_n^s.
\end{align*}
Notice that 
\[\inf_{D_{n+1}}(M_n-v)+\inf_{D_{n+1}}(v-m_n)=(M_{n}-m_{n})-\underset{D_{n+1}}{\rm osc}\,v\]
and by the inductive hypothesis \eqref{osc1},
\[\sup_{D_n}(M_n-v)+\sup_{D_n}(v-m_n)\le 2(M_{n}-m_{n}).\]
Now fix $\eps=\sigma/4$ and, recalling that $|m_n|,|M_n|\le \tilde C_\Omega$,  we get
\begin{align*}
\sigma(M_n-m_n) &\le \Big(1+\frac{\sigma}{2}\Big)(M_n-m_n)-\underset{D_{n+1}}{\rm osc}\,v \\
&\quad+ C\big[1+{\rm tail}_{p-1}((v-M_n)_+,R_n)+{\rm tail}_{p-1}((m_n-v)_+,R_n)\big]R_n^\frac{s}{p-1} \\[3pt]
&\quad + C\big[{\rm tail}_1((v-M_n)_+,R_n)+{\rm tail}_1((m_n-v)_+,R_n)\big]R_n^s,
\end{align*}
for some $C>1$ depending on $N$, $p$, $s$ and $\Omega$. Rearranging and using \eqref{osc1}, we get
\beq\label{osc6}
\begin{split}
\underset{D_{n+1}}{\rm osc}\,v &\le \Big(1-\frac{\sigma}{2}\Big)\mu R_{n}^{\alpha_{1}}+C\big[1+{\rm tail}_{p-1}((v-M_n)_+,R_n)+{\rm tail}_{p-1}((m_n-v)_+,R_n)\big]R_n^\frac{s}{p-1} \\
&\quad+ C\big[{\rm tail}_1((v-M_n)_+,R_n)+{\rm tail}_1((m_n-v)_+,R_n)\big]R_n^s.
\end{split}
\eeq
Now we need to estimate the tail terms. We note that for all $x\in D_i\setminus D_{i+1}$, $i\in\{0,\ldots n-1\}$, by \eqref{osc1} and monotonicity of the sequences $\{m_n\}$, $\{M_n\}$ we have
\[m_n-v(x)\le m_n-m_i \le (m_n-M_n)+(M_i-m_i) \le \mu(R_i^{\alpha_1}-R_n^{\alpha_1}).\]
Using $|m_{n}|, |M_{n}|, \|v\|_{L^{\infty}(\Omega)}\le \tilde C_{\Omega}$, for all $q\ge 1$ we have
\begin{align*}
\int_{\Omega\cap B_n^c}\frac{(m_n-v(x))_+^q}{|x|^{N+s}}\,dx &\le \int_{\Omega\cap B_0^c}\frac{(m_n-v(x))_+^q}{|x|^{N+s}}\,dx+\sum_{i=0}^{n-1}\int_{D_i\setminus D_{i+1}}\frac{(m_n-v(x))_+^q}{|x|^{N+s}}\,dx \\
&\le C+\mu^q\sum_{i=0}^{n-1}\int_{D_i\setminus D_{i+1}}\frac{(R_i^{\alpha_1}-R_n^{\alpha_1})^q}{|x|^{N+s}}\,dx \\
&\le C+C\mu^q\sum_{i=0}^{n-1}\frac{(R_i^{\alpha_1}-R_n^{\alpha_1})^q}{R_i^s} \le C+C\mu^qS_q(\alpha_1)R_n^{q\alpha_1-s},
\end{align*}
where we have set
\[S_q(\alpha_1)=\sum_{j=1}^\infty\frac{(8^{\alpha_1 j}-1)^q}{8^{sj}}.\]
Recalling the definition \eqref{tail} and setting $q=p-1$, we get (by convexity)
\begin{align*}
{\rm tail}_{p-1}((m_n-v)_+,R_n)R_n^\frac{s}{p-1} &\le C\big(1+\mu^{p-1}S_{p-1}(\alpha_1)R_n^{(p-1)\alpha_1-s}\big)^\frac{1}{p-1}R_n^\frac{s}{p-1} \\
&\le CR_n^\frac{s}{p-1}+C\mu S_{p-1}^\frac{1}{p-1}(\alpha_1)R_n^{\alpha_1},
\end{align*}
while for $q=1$ we immediately get
\[{\rm tail}_1((m_n-v)_+,R_n)R_n^s \le CR_n^s+C\mu S_1(\alpha_1)R_n^{\alpha_1}.\]
Similarly we prove the estimates
\[{\rm tail}_{p-1}((v-M_n)_+,R_n)R_n^\frac{s}{p-1}\le CR_n^\frac{s}{p-1}+C\mu S_{p-1}^\frac{1}{p-1}(\alpha_1)R_n^{\alpha_1},\]
\[{\rm tail}_1((v-M_n)_+,R_n)R_n^s\le CR_n^s+C\mu S_1(\alpha_1)R_n^{\alpha_1}.\]
Plugging these estimates into \eqref{osc6}, and recalling that $R_0<1$, we get
\beq\label{osc7}
\begin{split}
\underset{D_{n+1}}{\rm osc}\,v &\le \Big(1-\frac{\sigma}{2}\Big)\mu R_n^{\alpha_1}+C\big(S_1(\alpha_1)+S_{p-1}^\frac{1}{p-1}(\alpha_1)\big)\mu R_n^{\alpha_1}+C\big(R_n^\frac{s}{p-1}+R_n^s\big) \\
&\le \Big(1-\frac{\sigma}{2}+CS_1(\alpha_1)+CS_{p-1}^\frac{1}{p-1}(\alpha_1)\Big)8^{\alpha_1}\mu R_{n+1}^{\alpha_1}+C\big(R_0^{\frac{s}{p-1}-\alpha_1}+R_0^{s-\alpha_1}\big)8^{\alpha_1} R_{n+1}^{\alpha_1}.
\end{split}
\eeq
We claim that for all $q\ge 1$
\beq\label{osc8}
\lim_{\alpha_1\to 0^+}S_q(\alpha_1)=0.
\eeq
Indeed, for all $\alpha_1\in\ ]0,s/q[$ we have
\[S_q(\alpha_1)\le\sum_{j=1}^\infty\frac{1}{8^{(s-\alpha_1 q)j}}<\infty,\]
while clearly $(8^{\alpha_1 j}-1)^q/8^{sj}\to 0$ as $\alpha_1\to 0^+$, for all $j\in\N$, so $S_q(\alpha_1)\to 0$ as well. Applying \eqref{osc8} with $q=1,p-1$ respectively, for all $\alpha_1>0$ small enough we have
\[\Big(1-\frac{\sigma}{2}+CS_1(\alpha_1)+CS_{p-1}^\frac{1}{p-1}(\alpha_1)\Big)8^{\alpha_1}<1-\frac{\sigma}{4},\]
while we may choose $\mu>1$ big enough to have
\[\Big(1-\frac{\sigma}{4}\Big)\mu+C\big(R_0^{\frac{s}{p-1}-\alpha_1}+R_0^{s-\alpha_1}\big)8^{\alpha_1}\le\mu,\]
so from \eqref{osc7} we have
\[\underset{D_{n+1}}{\rm osc}\,v\le\mu R_{n+1}^{\alpha_1}.\]
Thus, we can find $m_{n+1},M_{n+1}\in[m_n,M_n]$ s.t.\
\[m_{n+1}\le\inf_{D_{n+1}}v\le\sup_{D_{n+1}}v\le M_{n+1}, \quad M_{n+1}-m_{n+1}=\mu R_{n+1}^{\alpha_1},\]
hence \eqref{osc1} holds at step $n+1$, which concludes the induction step.  For any $r\in\ ]0,R_0[$ there exists $n\in\N$ s.t.\ $R_{n+1}<r\le R_n$, so we have
\[\underset{D_r}{\rm osc}\,v\le \underset{D_{n+1}}{\rm osc}\,v\le\mu 8^{\alpha_1} r^{\alpha_1}.\]
Setting $C_8=\mu 8^{\alpha_1}$, we have
\[\underset{D_r}{\rm osc}\,\frac{u}{\ds}\le C_8r^{\alpha_1}.\]
Finally, for any $x_1\in\partial\Omega$ and an arbitrary $K>0$ in \eqref{bound}, translation invariance and homogeneity of $\fpl$ yield the conclusion. 
\end{proof}

\noindent
Our final steps require a technical lemma, which is contained in the proof of \cite[Theorem 1.2]{RS}:

\begin{lemma}\label{3con}
Let $\partial\Omega$ be $C^{1,1}$. If  $v\in L^\infty(\Omega)$ satisfies the following conditions:
\begin{enumroman}
\item\label{3con1} $\|v\|_{L^\infty(\Omega)}\le C_9$;
\item\label{3con3} for all $x_1\in\partial\Omega$, $r>0$ we have $\displaystyle\underset{D_r(x_1)}{\rm osc}\,v\le C_9r^{\beta_1}$;
\item\label{3con2} for all $x_0\in\Omega$ with ${\rm d}_\Omega(x_0)=R$, $v\in C^{\beta_2}(B_{R/2}(x_0))$ with $[v]_{C^{\beta_2}(B_{R/2}(x_0))}\le C_9(1+R^{-\mu})$,
\end{enumroman}
for some $C_{9}, \mu>0$ and $\beta_{1}, \beta_{2}\in\ ]0, 1[$, then there exist $\alpha\in\ ]0,1[$, $C_{10}>0$ depending on the parameters and $\Omega$ s.t.\ $v\in C^{\alpha}(\overline\Omega)$ and $[v]_{C^{\alpha}(\overline\Omega)}\le C_{10}$.
\end{lemma}

\noindent
Now we can prove our main result.
\vskip4pt
\noindent
{\em Proof of Theorem \ref{main}.} Let $u\in\w$, $f\in L^\infty(\Omega)$ satisfy \eqref{dp}, and set $K=\|f\|_{L^\infty(\Omega)}$, so $u$ satisfies \eqref{bound}. By homogeneity we can assume $K=1$. Let us collect some known facts about $u$. From \cite[Theorem 1.1]{IMS1} we know that there exist $\alpha_2\in\ ]0,s]$, $C>0$ s.t.\ $u\in C^{\alpha_2}(\overline\Omega)$ and
\beq\label{main1}
\|u\|_{C^{\alpha_2}(\overline\Omega)}\le C
\eeq
(in what follows, all constants depend on $N$, $p$, $s$, and $\Omega$), in particular $\|u\|_{L^\infty(\Omega)}\le C$. Besides, from \cite[Corollary 5.5]{IMS1} we know that for all $x_0\in\Omega$ with $R={\rm d}_\Omega(x_0)$
\beq
\label{main2}
\begin{split}
[u]_{C^{\alpha_2}(B_{R/2}(x_0))} &\le \frac{C}{R^{\alpha_2}}\Big[R^{p's}+\|u\|_{L^\infty(\Omega)}+R^{p's}\Big(\int_{B_R^c(x_0)}\frac{|u(y)|^{p-1}}{|x_0-y|^{N+ps}}\,dy\Big)^\frac{1}{p-1}\Big] \\
&\le \frac{C}{R^{\alpha_2}}\Big[R^{p's}+1+R^{p's}\Big(\int_{B_R^c(x_0)}\frac{1}{|x_0-y|^{N+ps}}\,dy\Big)^\frac{1}{p-1}\Big] \le \frac{C}{R^{\alpha_2}},
\end{split}
\eeq
since $R\le{\rm diam}(\Omega)$. Finally, from \cite[p.\ 292]{RS} we know that, with the same choice of $x_0$ and $R$ as above, the following estimate can be obtained by interpolation:
\beq\label{main3}
[{\rm d}_\Omega^{-s}]_{C^{\alpha_2}(B_{R/2}(x_0))}\le \frac{C}{R^{s+\alpha_2}}.
\eeq
Now we set $v=u/\ds$, and aim at applying Lemma \ref{3con} to this function. First, from \cite[Theorem 4.4]{IMS1} we know that $v\in L^\infty(\Omega)$ with
\beq\label{main4}
\|v\|_{L^\infty(\Omega)}\le C.
\eeq
Further, chosen $x_0\in\Omega$, $R={\rm d}_\Omega(x_0)$, we have for all $x,y\in B_{R/2}(x_0)$
\[
\begin{split}
\frac{|v(x)-v(y)|}{|x-y|^{\alpha_2}} &\le \frac{|u(x){\rm d}_\Omega^{-s}(x)-u(y){\rm d}_\Omega^{-s}(x)|}{|x-y|^{\alpha_2}}+\frac{|u(y){\rm d}_\Omega^{-s}(x)-u(y){\rm d}_\Omega^{-s}(y)|}{|x-y|^{\alpha_2}} \\[2pt]
&\le [u]_{C^{\alpha_2}(B_{R/2}(x))}\|{\rm d}_\Omega^{-s}\|_{L^\infty(B_{R/2}(x_0))}+\|u\|_{L^\infty(\Omega)}[{\rm d}_\Omega^{-s}]_{C^{\alpha_2}(B_{R/2}(x_0))} \\[2pt]
&\le \frac{C}{R^{\alpha_2}}\Big(\frac{2}{R}\Big)^s+\frac{C}{R^{s+\alpha_2}} \le \frac{C}{R^{s+\alpha_2}},
\end{split}
\]
for some $C>0$. Here we have used \eqref{main1}, \eqref{main2}, and \eqref{main3}. Finally, let $x_1\in\partial\Omega$ and $r>0$, and $\alpha_1\in(0,s]$, $C_8>0$, and $R_0\in(0,\rho/4]$ be as in Theorem \ref{osc}. We distinguish two cases:
\begin{itemize}[leftmargin=0.6cm]
\item[$(a)$] If $r\in(0,R_0)$, then by Theorem \ref{osc} we have
\[\underset{D_r(x_1)}{\rm osc}v\le C_8r^{\alpha_1}.\]
\item[$(b)$] If $r\ge R_0$, then by \eqref{main4} we have
\[\underset{D_r(x_1)}{\rm osc}v\le 2\|v\|_{L^\infty(\Omega)} \le \frac{C}{R_0^{\alpha_1}}r^{\alpha_1}.\]
\end{itemize}
In both cases, we can find $C>0$ s.t.\
\[\underset{D_r(x_1)}{\rm osc}v\le Cr^{\alpha_1} \quad \text{for all $r>0$}.\]
Then, hypotheses \ref{3con1}, \ref{3con2}, \ref{3con3} of Lemma \ref{3con} hold with $C_9=C$, $\beta_1=\alpha_1$, $\beta_2=\alpha_2$, and $\mu=\alpha_2+s$. Thus, we conclude that $v\in C^\alpha(\overline\Omega)$ and $[v]_{C^\alpha(\overline\Omega)}\le C$, which by \eqref{main4} implies $\|v\|_{C^\alpha(\overline\Omega)}\le C$, for $\alpha\in\ ]0,s]$ and $C>0$ only depending on $N$, $p$, $s$, and $\Omega$.
 \qed
\vskip5pt
\noindent
{\small {\bf Acknowledgement.} All authors are members of GNAMPA (Gruppo Nazionale per l'Analisi Matematica, la Probabilit\`a e le loro Applicazioni) of INdAM (Istituto Nazionale di Alta Matematica 'Francesco Severi'). This work was supported by the GNAMPA research project {\em Regolarit\`a, esistenza e propriet\`a geometriche per le soluzioni di equazioni con operatori frazionari non lineari} (2017). A.\ Iannizzotto is also supported by the research project {\em Integro-differential Equations and nonlocal Problems}, funded by Fondazione di Sardegna (2017). S. Mosconi  is also supported by the grant PdR 2016-2018 - linea di intervento 2: {\em Metodi Variazionali ed Equazioni Differenziali} of the University of Catania.
The authors would also like to thank Giuseppe Mingione for bringing their attention on some results contained in \cite{KMS,KMS2}.}

\end{document}